\documentclass[12pt]{amsart}
\usepackage{graphicx}
\newtheorem{thm}{Theorem}[section]
\newtheorem{assumption}[thm]{Assumption}
\newtheorem{cor}[thm]{Corollary}
\newtheorem{lem}[thm]{Lemma}
\newtheorem{prop}[thm]{Proposition}
\theoremstyle{definition}

\theoremstyle{remark}
\newtheorem{rem}[thm]{Remark}
\numberwithin{equation}{section}

\def\limn{\underset{n\rightarrow\infty}{\lim}} \def\limy{\underset{y\rightarrow\infty}{\lim}}
\def\limu{\underset{u\rightarrow\infty}{\lim}}
\def\limuz{\underset{u\rightarrow 0+}{\lim}}
\def\eqd{\overset{{\rm d}}{=}}

\def\vi{u_i}

\def\vin{v_i^n}

\def\sigmaf{\sigma_f}

\def\F{\mathcal{F}}

\def\mun{\mu^n}

\def\muln{\mu_l^n}   \def\mukn{\mu_k^n}

\def\muln{\mu_l^n}

\def\muthn{\mu_\theta^n}

\def\muzn{\mu_0^n}

\def\pt{\tilde{p}}

\def\tautn{\tilde{\tau}^n}
\def\taut{\tilde{\tau}}
\def\tautnD{\tilde{\tau}^n_D}
\def\tautD{\tilde{\tau}_D}

\def\hn{\hat{\mu}_l^n}
\def\cin{c^n_{i,l}}

\def\inN{\in\mathbb{N}}

\def\q{\boldsymbol{\pi}} \def\p{\textbf{p}}

\def\vphi{\boldsymbol{\varphi}}
\def\vphis{\boldsymbol{\varphi_\sigma}}
\def\vphiasf{\boldsymbol{\varphi_{\frac{\sqrt{\alpha}}{\sigmaf}}}} \def\vphihat{\boldsymbol{\hat{\varphi}}}
\def\vphit{\boldsymbol{\tilde{\varphi}}} \def\vphitn{\boldsymbol{\tilde{\varphi}^n}}

\def\pis{\boldsymbol{\pi_\sigma}}
\def\piasf{\boldsymbol{\pi_{\frac{\sqrt{\alpha}}{\sigmaf}}}}
\def\piasv{\boldsymbol{\pi_{\sqrt{\alpha}\sigma_v}}}

\def\p{\boldsymbol{\pi}} \def\pn{\boldsymbol{\pi^n}}
\def\zetatn{\boldsymbol{\tilde{\zeta}^n}}

\def\xitn{\boldsymbol{\tilde{\xi}^n}}
\def\chitn{\boldsymbol{\tilde{\chi}^n}}

\def\ptn{\boldsymbol{\tilde{\pi}^n}}
\def\pn{\boldsymbol{\pi^n}}
\def\pt{\boldsymbol{\tilde{\pi}}}
\def\phat{\boldsymbol{\hat{\pi}}}

\def\IS{I_S}

\renewenvironment{proof}[1][Proof]{\textbf{#1.} }{\ \rule{0.5em}{0.5em}}

\begin{document}




\title[The Proper Way to Use Bayesian Posterior Processes with Brownian Noise]{Parameter Estimation: The Proper Way to Use Bayesian Posterior Processes with Brownian Noise}

\author{Asaf Cohen} \thanks{Department of Electrical Engineering,
Technion--Israel Institute of Technology,
Haifa 32000, Israel,
shloshim@gmail.com, web https://sites.google.com/site/asafcohentau/}


\begin{abstract}
This paper studies a problem of Bayesian parameter estimation for a sequence of
scaled counting processes whose weak limit is a Brownian motion with an unknown drift. The main result of the paper is that the limit of the
posterior distribution processes is, in general, not equal to the posterior distribution process of the mentioned Brownian motion with the unknown drift.
Instead, it is equal to the posterior distribution process associated with a Brownian motion with  the same unknown drift and a different standard deviation coefficient. The difference between the two standard deviation coefficients can be arbitrarily large.
The characterization of the limit of the posterior distribution processes is then applied to a family of stopping time problems. We show that the proper way to find asymptotically
optimal solutions to stopping time problems w.r.t.~the scaled counting processes is by looking at the limit of the posterior distribution processes rather than by the naive approach of looking at the limit of the scaled counting processes themselves. The difference between the performances can be arbitrarily large.
\end{abstract}



\maketitle

\section{Introduction}\label{sec:introduction}
Brownian\footnote{This paper is an extended version of the paper with the same title that appears on \emph{Mathematics of Operations Research}. The only difference is that in this version we allow the case that thw system was activated before time $t=0$.} motion is a fundamental process in modeling various
stochastic phenomena. It has practical applications in various fields, such as mathematical finance, physics, queueing networks,
and signal processing. Brownian
motion is the continuous-time analogue of random walks and it can be obtained as the weak limit of discrete processes.

%
%
In this paper we study the relation between a Brownian motion with an unknown drift and a sequence of scaled counting processes in continuous time, which we term as `discrete processes'.
We assume that there exists a random variable $\theta$
with a known prior distribution, and a sequence of discrete processes
$\{(\tilde{L}^n_\theta(t))\}_{n\inN}$ that converges in distribution to
$\tilde{L}(t)=\tilde{L}_\theta(t):=\theta t+\sigma W(t)$, where $(W(t))$ is a standard
Brownian motion independent of the drift $\theta$. The decision maker (DM) does
not observe the random variable $\theta$, but rather observes continuously
$\tilde{L}^n:=\tilde{L}^n_\theta$.
Therefore, for sufficiently large $n\inN$, the observed process is
approximately distributed as a Brownian motion with an unknown
drift. For every $n$, define $\ptn$ (resp.
${\boldsymbol{\tilde{\pi}}}$) to be the (Bayesian) posterior
distribution process of $\theta$ given the observations from $\tilde{L}^n$ (resp.
$\tilde{L}$).

In many optimal control/stopping time problems such as the Bayesian
sequential testing problem in its different versions 
and the Bayesian Brownian bandit problem 
(see the literature review below) it is possible to formulate both the
problem and the solution by using the posterior distribution process. Because in these models the posterior distribution process is of interest, the naive approach of using results taken from optimal stopping problems w.r.t.~the posterior distribution process
$\boldsymbol{\tilde{\pi}}$, such as the structure of the optimal strategy, and implementing them in optimal stopping problems concerning the
process $\tilde{L}^n$ (for sufficiently large $n$) is not relevant;
the right approach should be to find the limit of the posterior
distribution processes $\ptn$ instead of the posterior distribution
process of the limit process $\tilde{L}=\limn \tilde{L}^n$. To
illustrate this point, in Remark \ref{rem:brown_present} below we show that $\tilde{L}(t)$, the
value of the process $(\tilde{L}(s))_{0\leq s\leq t}$ at time $t$, is
a sufficient statistic for the posterior distribution process
$\boldsymbol{\tilde{\pi}}$ at time $t$. That is,
$\boldsymbol{\tilde{\pi}}$ is independent of past observations from $\tilde{L}$, given the
present value of $\tilde{L}$.
%
%
%
%
%
%
%
However, it appears that, usually, $\ptn$ depends not only on the present value of $\tilde{L}^n$, but
also on past observations from $\tilde{L}^n$. Therefore, it uses `more information' than $\pt$ does and it is `more accurate'. We show below that this is indeed the case.
\subsection{Main Results}
The paper's main results are: (1) characterizing the limit of the posterior processes, $\limn\ptn$, and (2) using this characterization in order to find asymptotically optimal solutions for Bayesian stopping time problems.
It might happen that $\limn\ptn$ is trivial. This case arises, e.g.,
when the value of $\theta$ is detected in an infinitesimal time
interval or when the limit is a constant. Under mild
assumptions, we find an explicit expression for the limit of the
posterior distribution processes, $\limn\ptn$, and show that in general $\limn\ptn\neq \pt$.
Although, the limit $\limn\ptn$ has a different distribution than the posterior distribution process $\pt$, we prove
that this limit can
be expressed as the posterior distribution process of a different Brownian motion with an unknown drift that is given by
\begin{align}\label{eq:hat_L2intro}
\hat{M}(t)=\hat{M}_\theta (t) :=\theta t + \sigma' W'(t), \;\;t\in[0,\infty),
\end{align}
where $(W'(t))$ is a Brownian motion independent of $\theta$ and
$0<\sigma'\leq \sigma$. The quantity $\sigma'$ depends on the
structure of the processes $\{\tilde{L}^n\}_{n\inN}$. Since
$\sigma'\leq \sigma$, the paths of the process $(\hat{M}(t))$ will be more
concentrated around the path of the linear drift $(\theta t)$ than the paths of the
process $(\tilde{L}(t))$. In other words, $(\hat{M}(t))$ is less noisy than
$(\tilde{L}(t))$. Therefore, it is easier to estimate the parameter $\theta$ given $(\hat{M}(t))$ than given $(\tilde{L}(t))$; that is,
$\limn\ptn$ is more informative than $\boldsymbol{\tilde{\pi}}$.
In addition, we identify when the equality $\sigma' = \sigma$ holds.
We show that it happens if and only if the processes
$\{(\tilde{L}^n_l(t))\}_{l\in S,\; n\inN}$ satisfy a memorylessness
property and no information, regarding the posterior distribution processes, is lost by looking at the present values of the $\{(\tilde{L}^n_l(t))\}_{l\in S,\; n\inN}$ rather than at their past and present values (e.g., Poisson processes with unknown rates that depend on $\theta$ and $n$). This is the same property that holds in the Brownian motion with an unknown drift model. We also show that the difference between the parameters $\sigma'$ and $\sigma$ can be arbitrarily large.
%
%
%
%
%
%
%
%
%

Our study thus strengthens the motivation for analyzing the posterior
distribution process of a Brownian motion with an unknown drift. Moreover, the fact
that the structure of $\limn\ptn$ is the same as that of
${\boldsymbol{\tilde{\pi}}}$ is interesting and raises further
questions about the structures of posterior processes of more general
diffusion processes that involve uncertainty.

We finally show how to find
asymptotically optimal solutions for the Bayesian stopping time problems
for $\tilde{L}^n$ by using the approximation $\limn\ptn$ rather than
${\boldsymbol{\tilde{\pi}}}$.
In fact, since the difference between $\sigma'$ and $\sigma$ can be arbitrarily large, by using the
incorrect approximation ${\boldsymbol{\tilde{\pi}}}$ in order to calculate the optimal strategy in the $n$-th model, the performance can be arbitrarily bad.

The rest of the paper is organized as follows: The introduction is concluded with a literature
review. In Section \ref{sec:Technical} we introduce some technical preliminaries.
In Section \ref{sec:auxiliary} we
present a model of a Brownian motion with an unknown drift. We give
a closed-form formula for the posterior distribution process. In Section \ref{sec:models} we
define a sequence of systems (indexed by $n\inN$) that converges to
a Brownian motion with an unknown drift. In Section \ref{sec:Limit_Processes} we present the
main results and find the distribution of the limit of the sequence
of the posterior distribution processes. In Section \ref{sec:optimal_stopping_time} we consider a general
optimal stopping problem for the $n$-th system and find asymptotically optimal solution by using
the presentation we give to the limit of the posterior distribution processes. Summary and directions for future research appear in Section \ref{sec:conclusion}. The Appendix contains the proofs of several
theorems.

\subsection{Literature Review}
The model of a DM who observes a Brownian motion
with an unknown drift (and known standard deviation) is well
explored in the literature and appears in the context of filtering
theory, optimal stopping problems, and economics.

A variation of this model was studied in filtering theory by Kalman
and Bucy (1961) \cite{Kalman1961} and Zakai (1969) \cite{Zakai1969}.
These authors analyzed a more general model, where a DM observes a function of
a diffusion process with an additional noise, which is formulated as
a Brownian motion. They provided equations that the posterior or the
unnormalized posterior distribution process satisfies.

Shiryaev (1978) \cite{Shiryaev1978} defined a Bayesian sequential
testing problem where a DM observes continuously a Brownian motion
with an unknown drift and has two hypotheses about the drift
together with a prior probability about these hypotheses. In this
problem the goal of the DM is to test sequentially the hypotheses
with a minimal loss. The choice that the DM should make is to choose a
stopping time and at that time to guess which one of the two hypotheses holds.
This problem was generalized in several ways. Zhitlukhin
and Shiryaev (2011) \cite{Zhitlukhin2011} generalized it to three
hypotheses. Gapeev and Peskir (2004) \cite{Gapeev2004} explored the
problem with finite horizon. Gapeev and Shiryaev (2011)
\cite{Gapeev2011} explored a sequential testing problem where the
observed process is a diffusion process satisfying a stochastic
differential equation. Buonaguidi and Muliere (2013) \cite{Buonaguidi2013} studied a sequential testing problem where the
observed process is a L\'{e}vy process with unknown parameters.

Berry and Friestedt (1985) \cite{Berry1985} investigated a Bayesian Brownian bandit
problem where a DM operates a two-armed bandit with two available
arms; a safe arm that yields a constant payoff, and a risky arm that
yields a stochastic payoff, which is a Brownian motion with an
unknown drift. There are two hypotheses about the drift together
with a prior probability about these hypotheses. The DM has to
decide when to switch from the risky arm to the safe arm. Bolton and
Harris (1999) \cite{Bolton1999} investigated a game involving this
type of bandit. Cohen and Solan (2013) \cite{Cohen2013} studied the
single DM problem in the case where the observed process is a
L\'{e}vy process with unknown parameters.

Other statistical Bayesian tests involving hypotheses on a Brownian
motion with an unknown drift can be found in the literature. For
example, Polson and Roberts (1994) \cite{Polson1994} investigated
the likelihood function for a diffusion process with an unknown
parameter and provided an example of a Brownian motion with an
unknown drift with a normal prior on the drift.

In economic theory, the model of a Brownian motion with two
prior hypotheses about the drift was studied, e.g., by Felli and
Harris (1996) \cite{Felli1996}, Bergemann and Valimaki (1997)
\cite{Bergemann1997}, Bolton and Harris (1999) \cite{Bolton1999},
Keller and Rady (1999) \cite{Keller1999}, and Moscarini (2005)
\cite{Moscarini2005}. In Jovanovic (1979) \cite{Jovanovic1979} the
prior about the drift is assumed to have the normal distribution. In
the listed papers it is assumed that random changes appear after
every small time interval and the process of total change can be
modeled approximately by a Brownian motion.

Another well-known example
of the use of Brownian motion as a continuous-time approximation of
a discrete-time processes is in queueuing theory; under heavy
traffic, the queue size, which changes by discrete jumps after every
random time interval, converges to a reflected Brownian motion with
a drift. The uncertainty about the drift can model a situation of a
$G/G/1$ queue in heavy traffic where the rate of service is unknown.
Such a case arises, for example, when the number of projects that a
server works on and the amount of the effort that it dedicates to
each project are unknown. For further examples of queueing models
with parameter uncertainty see Whitt (2006) \cite{Whitt2006} and
the references therein.
\section{Technical Preliminaries}\label{sec:Technical}
Let $T>0$ and let $\mathcal{D}_T:=\mathcal{D}[0,T]$ (resp.
$\mathcal{D}_\infty:=\mathcal{D}[0,\infty)$) be the space of
real-valued RCLL (right-continuous with left limits) functions on
$[0,T]$ (resp. $[0,\infty)$). 
%
%

Fix a Borel set $S\subseteq\mathbb{R}$. Let $\mathcal{E}_T$ (resp.
$\mathcal{E}_\infty$) be the space of real-valued functions on
$S\times[0,T]$ (resp. $S\times[0,\infty)$) that are $\mathcal{D}_T$
(resp. $\mathcal{D}_\infty$) with respect to the second variable.
The space $\mathcal{E}_T$ is endowed with the metric\footnote{All
the limiting functions in this paper are in
$\mathcal{C}_\infty:=\mathcal{C}[0,\infty)$ or
$\mathcal{C}_T:=\mathcal{C}[0,T]$ (the subspaces of continuous
functions on $[0,\infty)$ and $[0,T]$, respectively) with respect to
their second variable. Therefore, the uniform topology is sufficient
for our purpose instead of the often used Skorokhod topology (see
Chen and Yao (2001, Ch.~5.1) \cite{Chen2001} for further
discussion).}
\begin{align}\label{e_T}
e_T(\nu,\kappa) &:= \underset{l\in S,
t\in[0,T]}{\sup}\,|\nu(l,t)-\kappa(l,t)| \wedge 1 ,\;\;\text{for
$\nu,\kappa\in\mathcal{E}_T$}.
\end{align}
By using this metric, we define on the space $\mathcal{E}_\infty$ the
metric
\begin{align}\notag
e_\infty(\nu,\kappa) &:= \sum_{T=1}^\infty
e_T(\nu,\kappa)\frac{1}{2^T} ,\;\;\text{for
$\nu,\kappa\in\mathcal{E}_\infty$}.
\end{align}
The metric $e_\infty$ is a generalization of the standard metric
with which one usually defines convergence to a Brownian motion (see
Karatzas and Shreve (1991) \cite{Karatzas1991}) for functions of two
variables.

\begin{rem}\label{rem:u.o.c}
Let $\{\kappa\}\cup\{\kappa^n\}_{n\inN}\subset \mathcal{E}_\infty$.
From the definitions of $e_T$ and $e_\infty$ it follows that
$\{\kappa^n\}_{n\inN}$ converges to $\kappa$ if and only if for
every $T\in\mathbb{N}$, the restriction of $\{\kappa^n\}_{n\inN}$ to
$S\times[0,T]$ converges to the restriction of $\kappa$ to
$S\times[0,T]$.
\end{rem}

Throughout the paper we denote processes with observations in
$\mathcal{E}_\infty$ by bold Greek letters, processes with
observations in $\mathcal{D}_\infty$ by capital Latin letters, and
functions from $S$ to $\mathbb{R}$ by small Latin letters.

\subsection{Types of Convergence}
Let $\{\boldsymbol{\zeta}\}\cup\{\boldsymbol{\zeta^n}\}_{n\inN}$ be
measurable mappings from a probability space
$(\Omega,\mathcal{F},P)$ to
$(\mathcal{E}_\infty,\mathcal{B}(\mathcal{E}_\infty))$. We define
two types of convergence $\limn \boldsymbol{\zeta^n} =
\boldsymbol{\zeta}$ that are used in this paper.

\subsubsection{\textbf{Uniform Convergence over Compact Sets.}}
We say that $\{\boldsymbol{\zeta}^n\}_{n\inN}$ \emph{converges
uniformly over compact sets} (u.o.c.) to $\boldsymbol{\zeta}$ if
\begin{align}\label{def:e_u.o.c.}
&P\left( \limn
e_\infty(\boldsymbol{\zeta^n},\boldsymbol{\zeta})=0\right)=1.
\end{align}
Remark \ref{rem:u.o.c} implies that
Eq.~(\ref{def:e_u.o.c.}) is equivalent to the requirement that for
every $T\in\mathbb{N}$ one has
\begin{align}\label{def:e_u.o.c.T}
&P\left( \limn
e_T(\boldsymbol{\zeta^n},\boldsymbol{\zeta})=0\right)=1.
\end{align}
\subsubsection{\textbf{Convergence in Distribution.}}
%
We
say that $\{\boldsymbol{\zeta^n}\}_{n\inN}$ \emph{converges in
distribution} to $\boldsymbol{\zeta}$ (and write $\limn
\boldsymbol{\zeta^n}\eqd \boldsymbol{\zeta}$) if for every bounded
and continuous function $f$ (w.r.t.~the metric $e_\infty$) defined
on $\mathcal{E}_\infty$ one has
\begin{align}\notag
&\limn E[f(\boldsymbol{\zeta^n})] = E[f(\boldsymbol{\zeta})].
\end{align}
As is well known, convergence u.o.c.~implies convergence in
distribution.

If $X:\Omega\rightarrow \mathcal{D}_\infty$ or
$h:S\rightarrow\mathbb{R}$, then one may look at $X(\omega)$ and $h$
as elements in $\mathcal{E}_\infty$ that are independent of the
first and second variables, respectively.

\section{An Auxiliary Model - Brownian Motion with an Unknown Drift}\label{sec:auxiliary}
\subsection{Formulations and Notations}

In this section we study a model of a Brownian motion with an
unknown drift. Let $\theta$ be a random variable with a
countable\footnote{The results in the section can be extended to a
Borel set $S$ with the cardinality of the continuum, see Section
\ref{sec:continuum}.} support $S\subset\mathbb{R}$ and a
distribution $\pi:=\{\pi_l\}_{l\in S}$.
Let $(W(t))$ be a standard Brownian motion independent of $\theta$.
Set $\sigma>0$ and define
$$X(t) = X_\theta(t):=\theta t + \sigma W(t), \;\;t\in[0,\infty).$$
Suppose that the DM observes the process $(X(t))$ continuously, but does not
observe $\theta$. The drift $\theta $ is not known by
the DM. For every $l\in S$ define the hypothesis $H_l:\;\theta=l$.
The parameter $\pi_l$ represents the prior probability that $H_l$ is
true. Denote by $P_l$ the probability measure over the space of
realized paths under the hypothesis $H_l$, and by $P:=P_{\pi} = \sum_{l\in S}\pi_l P_l$ the
probability measure that corresponds to the description above (see  Gapeev and Peskir (2004)
\cite{Gapeev2004} for a rigorous construction of $P$).

\subsection{The Posterior Distribution Process}
At time $t=0$, the parameter
$\theta$ is chosen randomly according to the distribution $\pi$. The DM does not observe
$\theta$ but he knows $\pi$ and and $\sigma$.
At each time instant $t$ the DM observes the process $(X(t))$ and
updates his belief about the hypotheses based on this information in
a Bayesian fashion. We would like to give a closed-form expression
to the \emph{posterior distribution process}\footnote{Note that the
process $(\q(l,t))$ depends on the prior distribution $\pi$; indeed,
for every $l\in S$ one has $\q(l,0)=\pi_l$. To save cumbersome
notation, we omit the dependence on $\pi$.}
\begin{align}\label{eq:pi}
\q(l,t) := P(\theta =l \mid \mathcal{F}_t^X ; \pi), \;\;l\in
S,\;t\in[0,\infty),
\end{align}
where $\mathcal{F}_t^X$ is the sigma-algebra that is generated by
$(X(s))_{0\leq s \leq t}$. The value $\q(l,t)$ is the posterior distribution
at time $t$ that $H_l$ is true given the past observations.

Without loss of generality we assume that $0\in S$, since by taking
$m\in S$ one can look at the process
$$X(t)-mt =(\theta-m) t + \sigma W(t), \;\;t\in[0,\infty).$$
The processes $(X(t)-mt)$ and
$(X(t))$ admit the same filtration; and $0$ is in the support of
$\theta-m$.

An important auxiliary process is the Girsanov process, also called
the Radon--Nikod\'{y}m density, which is defined  by
\begin{align}\label{varphi}
\vphi(l,t) :=
\frac{d(P_l\mid\mathcal{F}_t^X)}{d(P_0\mid\mathcal{F}_t^X
)},\;\;l\in S,\;t\in[0,\infty).
\end{align}
The next result connects the process $\q$ to the process $\vphi$.
\begin{lem}
For every $l\in S$, and every $t\in[0,\infty)$,
\begin{align}\label{frac_q_l_q_m}
\q(l,t) = \frac{\pi_l\vphi(l,t)}{\sum_{k\in S}
 \pi_k\vphi(k,t)}.
\end{align}
\end{lem}
For a proof, see Cohen and Solan (2013, Lemma 1) \cite{Cohen2013}.
By Jacod and Shiryaev (1987, Ch.~III, Theorems 3.24 and 5.19)
\cite{Jacod1987} the process $\vphi$ admits the following
representation:
\begin{align}\notag
\vphi(l,t) &=\exp \left\{ \frac{l}{\sigma^2} X(t)-
\frac{1}{2}\left(\frac{l}{\sigma}\right)^2 t\right\}
\\\label{varphi_Z} &= \exp \left\{ \frac{l}{\sigma} W(t)-
\frac{1}{2}\left(\frac{l}{\sigma}\right)^2
t+\frac{\theta}{\sigma}\cdot \frac{l}{\sigma} t\right\},\;\;l\in
S,\;t\in[0,\infty).
\end{align}

\begin{rem}\label{rem:brown_present}
Notice that, based on the observed process $(X(s))_{0\leq s\leq t}$, the present value at time $t$,
$X(t)$, is a sufficient statistic for $\theta$.
That is, for every $l\in S$
and every $t\in[0,\infty)$, the value of the process $(\vphi(l,s))$
at time $t$, $\vphi(l,t)$, and therefore also $\q(l,t)$, depends on
the process $(X(s))_{0\leq s\leq t}$ only through $X(t)$. This means
that the Radon--Nikod\'{y}m density and the posterior distribution
process at time $t$ depend on $(X(s))$ through the present value
$X(t)$ and are independent of past values $(X(s))_{0\leq s< t}$.
\end{rem}


In order to emphasize the dependence of the processes $\vphi$ and
$\q$ on $\sigma$, we denote them by $\vphis$ and $\pis$.

%

\section{Deterministic and Random Parameter Systems}\label{sec:models}
In this section we define a sequence of processes indexed by $n\inN$
that converges in distribution (w.r.t. $n$) to a Brownian motion
with an unknown drift. For each such process we define a relative
posterior distribution process. In Section \ref{sec:DPS} we define a
model of a system that consists of arrivals with a known rate. In
Section \ref{sec:RPS} we generalize the model to a system that
consists of arrivals with an unknown rate. In Section \ref{sec:RPSn}
we define a sequence of systems with unknown rates. In Section
\ref{sec:Posterior} we show that under proper assumptions, the
scaled number of arrivals to these systems can be approximated
by a Brownian motion with an unknown drift.

\subsection{Deterministic Parameter System}\label{sec:DPS}

We define a system that consists of arrivals (each of size $1$) that
occur according to the random variables $\{v_i\}_{i\geq 1}$. We
assume that the system was activated before time $t=0$. The
parameter $t_v$ is the time passed since the last arrival that
occurred before time $t=0$. We start the numeration of arrivals from
time $t=0$.
 $v_1$ is interpreted as the time passed from $t=0$
until the first arrival; and for every $i\geq 2$, the random
variable $v_i$ is interpreted as the interarrival time between the
$(i-1)$-th and the $i$-th arrivals into the system. We present the
interarrival time distribution as $\frac{v}{\mu}$, where $v$ is a
nonnegative random variable with expectation $1$ and $\mu$ is a
positive constant.

Formally, a \emph{deterministic parameter system}
$$\mathcal{S}=\left( t_v, v, \mu, \{v_i\}_{i\geq 1}\right)$$ is given by
$\\\bullet\;\;$ a nonnegative constant $t_v$;
$\\\bullet\;\;$ a nonnegative random variable $v$;
$\\\bullet\;\;$ a positive constant $\mu$;
$\\\bullet\;\;$ a sequence of independent random variables
$\{v_i\}_{i\geq 1}$.

We make the following assumption on $\{v_i\}_{i\geq 1}$. 
\begin{assumption}\label{asu:t_u_t_v}
$\\ $For every $t\in[0,\infty)$, one has $P(v_1+t_v\geq
t)=P\left(\frac{1}{\mu}v\geq t \mid\frac{1}{\mu}v\geq t_v\right)$,
and for every $i\geq 2$, $v_i$ is distributed as the random variable
$\frac{1}{\mu}v$.
\end{assumption}
%
%
We assume that $v$ has a finite variance and without loss of generality,
we assume that it has expectation $1$.
\begin{assumption}\label{asu:expectation1}
$\\ $\ref{asu:expectation1}.1. $E[v]=1$.
$\\ $\ref{asu:expectation1}.2. $\sigma_v^2:= {\rm Var}[v]<\infty$.
\end{assumption}
The \emph{arrival rate} is defined by $\frac{1}{E[v_2]}$, which, by
the definition of $v_i$ and Assumption \ref{asu:expectation1}.1,
equals $\mu$.

\subsubsection{The Counting Processes}
Define the process
\begin{align}\label{eq:L_def}
L(t)&:= \max  \left\{m\left|\; \sum_{i=1}^m v_i \leq t
\right\}\right., \;\;t\in[0,\infty).
\end{align}
The process $(L(t))$ counts the number of arrivals during the time
interval $[0,t]$, and it is called the \emph{counting process} of the
system.

\subsection{Random Parameter System}\label{sec:RPS}
Let $\theta$ be a random variable with bounded and
countable\footnote{All the results in the paper can be extended to a
bounded set $S$ with cardinality of the continuum, see Subsection
\ref{sec:continuum}.} support $S\subseteq\mathbb{R}$.
 For every $l\in\mathcal{S}$, let $\pi_l:=
P (\theta=l)$. 
 Consider a
constant $t_v$ and a random variable $v$ that satisfy Assumptions
\ref{asu:t_u_t_v} and \ref{asu:expectation1}, respectively. For
every $l\in S$, let
$$\mathcal{S}_l=\left(t_v, v, \mu_l, \{v_{i,l}\}_{i\geq 1} \right)$$ be a
deterministic parameter system such that the random variables
$\{v_{i,l}\}_{i\geq 1}$ are independent of $\theta$. Let $(L_l(t))$
be the corresponding counting process. A \emph{random parameter
system} is a system where the parameter $\mu$ is chosen randomly
according to $\theta$. That is, it is a random variable, and its
support is the collection of the deterministic parameter systems
$\{\mathcal{S}_l\}_{l\in S}$. Formally, a random parameter system
$$\mathcal{RS}_{\pi}(\theta)=\left(t_v, v, \mu_\theta, \{v_{i}\}_{i\geq 1}, \pi \right)$$ is given by\footnote{Note that the
sequence $\{v_i\}_{i\geq 1}$ depends on the random variable
$\theta$; indeed, for every $i\geq 1$, one has $v_i=v_{i,\theta}
=\sum_{l\in S}\mathbb{I}_{\{\theta=l\}}v_{i,l} $. To avoid
cumbersome notation, we omit the dependence on $\theta$.}
$$\left(t_v, v, \mu_\theta, \{v_i\}_{i\geq 1} \right) =
\sum_{l\in S}\mathbb{I}_{\{\theta=l\}}\left(t_v, v, \mu_l,
\{v_{i,l}\}_{i\geq 1} \right),$$ where $\pi:=\{\pi_l\}_{l\in S}$.
The corresponding counting process is $$L(t) = L_\theta (t)
=\sum_{l\in S}\mathbb{I}_{\{\theta=l\}}L_l(t), \;\;t\in[0,\infty).$$

A DM operates a random parameter system. The parameter $\theta$
represents the type of the arrival rate and it is unknown to the DM.
For every $l\in S$, the parameter $\pi_l$ represents the probability
that the arrival rate's type is $\theta=l$.

For every $l\in S$, define the hypothesis $H_l:\;\theta=l$. Denote
by $P_l$ the probability measure over the space of realized paths
under the hypothesis $H_l$, and by $P:=P_{\pi} = \sum_{l\in S}\pi_l P_l$ the probability
measure that corresponds to the description above.

\subsubsection{The Posterior Distribution Processes}

At time $t=0$, the DM observes the initial state $(t_v, \pi)$
without observing $\theta$, and thereafter he observes the counting process
$(L(t))$ continuously. At each time instant $t$, the DM can update
his belief on $\theta$ in a Bayesian fashion. Formally, the
\emph{posterior distribution process} is
\begin{align}
\p (l,t) :&=P(\theta = l \mid \mathcal{F}_{t}^{L}, \pi) = P(\theta =
l \mid L(t), t_v, v_1, \ldots, v_{L(t)}; t; \pi), \;\;l\in
S,\;t\in[0,\infty),
\end{align}
where $\mathcal{F}^{L}_t$ is the sigma-algebra generated by
$(L(s))_{0\leq s \leq t}$. This is the posterior distribution process at
time $t$ that $H_l$ is true given past observations of interarrivals
times from the system $v_1, \ldots, v_{L(t)}$, and the absence of
arrivals during the time interval
$\left(\sum_{i=1}^{L(t)}v_i,t\right]$. That is, the DM updates his
belief using all the available information he has from the observed
process up to time $t$.

\subsection{The $n$-th System}\label{sec:RPSn}
In this section we define a sequence of random parameter systems
indexed by a parameter $n$, which can be any natural number. All the
notation established in Section \ref{sec:RPS} is carried forward,
except that we append a superscript $n$ to denote a quantity which
depends on $n$. We assume that the random variables $v$ and $\theta$
are independent of $n$.

For every $n\in\mathbb{N}$, let
$$\mathcal{RS}^n_{\pi}(\theta)=\left(t^n_v, v, \mun_\theta,
\{\vin\}_{i\geq 1}, \pi \right)$$ be a sequence of random parameter
systems with the corresponding counting process $$L^n(t)=L^n_\theta(t) =
\sum_{l\in S}\mathbb{I}_{\{\theta=l\}}L^n_l(t),
\;\;t\in[0,\infty).$$
In order to define the diffusion approximation, we investigate the
$n$-th system at time $nt$. Without loss of generality we assume
that $0\in S$, since by taking $m\in S$ one can look at the random
variable $\theta - m$, and $0$ belongs to its support.
For every $n\in\mathbb{N}$ define the scaled posterior distribution process
\begin{align}\label{eq:tilde_pi}
\ptn(l,t):= \pn(l,nt), \;\;l\in S,\;t\in[0,\infty).
\end{align}

\subsection{The Posterior Distribution Process of the Limit of the Counting Processes}\label{sec:Posterior}
In this section we find a diffusion approximation related to the
sequence of processes $\{L^n\}_{n\inN}$. To this end, we require
that the rates under the different types are relatively
close, up to order of\footnote{In Remarks \ref{rem1} and \ref{rem2} we
explain why we require an order of $\frac{1}{\sqrt{n}}$ and detail
the differences in the analysis in case that the order is higher or smaller
than $\frac{1}{\sqrt{n}}$.} $\frac{1}{\sqrt{n}}$.
Loosely speaking, it states that $\muthn\approx \alpha + \frac{1}{\sqrt{n}}\theta$.
It reminds the heavy traffic condition, which asserts that
the difference between the arrival rate and the departure rate in a
G/G/1 queue is by order of $\frac{1}{\sqrt{n}}$.

For every $n\inN$,
let $h^n:S\rightarrow \mathbb{R}$ be the function
\begin{align}\label{eq:h}
h^n(l):=\sqrt{n}(\muln-\muzn).
\end{align}
\begin{assumption}\label{asu:lambda_rates}
$\\ $\ref{asu:lambda_rates}.1. $\limn \underset{l\in
S}{\sup}\,|h^n(l)-l| = 0 $.
$\\ $\ref{asu:lambda_rates}.2. $\limn \muzn = \alpha$, where $\alpha$
is a positive constant.
\end{assumption}
In Remark \ref{rem1} below we discuss about the necessity of this assumption.
Assumption \ref{asu:lambda_rates}.1 relates\footnote{Assumption
\ref{asu:lambda_rates}.1 can be written also as $\limn
e_\infty(h^n,\IS)=0$, where $\IS$ is the identity function on $S$.} to the difference between the arrival rates
under the different types. It states that every two possible arrival rates
are distinguished by an order of $\frac{1}{\sqrt{n}}$. Assumption
\ref{asu:lambda_rates}.2 states that the limit of the sequence of
rates $\{\muzn\}_{n\inN}$ is positive. Together with Assumption
\ref{asu:lambda_rates}.1 it implies that
\begin{align}\label{eq:A2}
\limn \underset{l\in S}{\sup}\,|\muln - \alpha |=\limn
\underset{l\in S}{\sup}\,|\muln - \muzn | = 0.
\end{align}
This assumption is also fundamental for the diffusion approximation
of the sequence of processes $\ptn$.
%
%
%
%

For every $n\inN$, denote
$$\check{L}^n(t):=\frac{L^n(nt)-\muthn nt}{\sqrt{n}},\;\;t\in[0,\infty).$$
The following result was proved, e.g., in Billingsley (1999, Theorem
14.6) \cite{Billingsley1999}.
\begin{prop}
Under Assumptions \ref{asu:t_u_t_v}, \ref{asu:expectation1}, and
\ref{asu:lambda_rates}.2, there exists a standard Brownian motion
process $(W(t))$, independent of $\theta$, such that
\begin{align}\label{eq:L_diffusion}
%
\limn \check{L}^n\eqd \sigma_v\sqrt{\alpha}W.
\end{align}
\end{prop}
%
%
Although the DM observes the process $(L^n(nt))$, he does not
observe $\theta$, and therefore does not observe $\mu_\theta$. That
is, the parameter $\mu_\theta$ is not known by the DM. Hence, the
sigma-algebra that is generated by the relative process
\begin{align}\notag
\frac{L^n(nt)-\muthn nt }{\sqrt{n}}, \;\;t\in[0,\infty),
\end{align}
is different from the sigma-algebra $\mathcal{F}^{L^n}_{nt}$. We
therefore define the relative process\footnote{Notice that $\muthn$
is replaced by $\muzn$.}
\begin{align}\notag
\tilde{L}^n(t):=\frac{L^n(nt)-\muzn nt }{\sqrt{n}},
\;\;t\in[0,\infty). \notag
\end{align}
For every $n\inN$ the process $(\tilde{L}^n(t))$ can be calculated by the DM, since the sigma-algebra that is generated by $(\tilde{L}^n(t))$ is
$\mathcal{F}^{\tilde{L}^n}_{t}:=\mathcal{F}^{L^n}_{nt}$, which is observed by the DM. The process
$(\tilde{L}^n(t))$ can be expressed as
\begin{align}\label{eq:Lnhat2}
\tilde{L}^n(t)=\frac{L^n(nt)-\muthn nt }{\sqrt{n}} +
\sqrt{n}(\muthn-\muzn) t , \;\;t\in[0,\infty).
\end{align}
From Eq.~(\ref{eq:L_diffusion}) the limit of the first term is a
Brownian motion with a standard deviation $\sqrt{\alpha}\sigma_v$ and
without drift, and from Assumption \ref{asu:lambda_rates}.1 the
limit of the second term is $(\theta t)$.
Therefore, the limit process:
\begin{align}\notag
&\tilde{L}(t):\eqd \limn\tilde{L}^n(t), \;\;t\in[0,\infty),
\end{align}
exists and it is a Brownian motion with an unknown drift. This is
summarized in the following proposition.

\begin{prop}\label{prop:hat_L}
Under Assumptions \ref{asu:t_u_t_v}, \ref{asu:expectation1}, and
\ref{asu:lambda_rates}, there exists a standard Brownian motion
process $(W(t))$, independent of $\theta$, such that
\begin{align}\label{eq:hat_L}
&\tilde{L}(t)= \theta t +\sqrt{\alpha}\sigma_v W(t),
\;\;t\in[0,\infty).
\end{align}
\end{prop}

From the definitions of $\vphis$ and $\pis$ (recall Eqs.~(\ref{eq:pi}) and (\ref{varphi}) and the notation given after Remark \ref{rem:brown_present}) we deduce the following corollary.
\begin{cor}\label{cor:advantage_of_posterior}
Let $\vphit $ and $\pt $ be the Radon--Nikod\'{y}m derivative
process and the posterior distribution process, respectively, under
$\mathcal{F}^{\tilde{L}}_t$. Under Assumptions \ref{asu:t_u_t_v},
\ref{asu:expectation1}, and \ref{asu:lambda_rates}, the process
$\vphit$ is distributed as
$\boldsymbol{\varphi}_{\sqrt{\alpha}\sigma_v}$ and the process
$\pt$ is distributed as
$\boldsymbol{\pi}_{\sqrt{\alpha}\sigma_v}$.
\end{cor}

The following remark explains the requirement that the proper
rates under the different types are relatively close, up to order of
$\frac{1}{\sqrt{n}}$ (Assumption \ref{asu:lambda_rates}.1).
\begin{rem}\label{rem1}
If there exists a parameter value $l^*\in S$ such that the difference
between the rates $\mu_{l^*}^n$ and $\muzn$ satisfies\footnote{Hereafter, the notation $|f_1(n)-f_2(n)|>>f_3(n)$ (resp.~$<<$) means that $\limn |f_1(n)-f_2(n)|/f_3(n) = \infty$ (resp. $0$).}
$|\mu_{l^*}^n - \muzn|>> \frac{1}{\sqrt{n}},$ then under $\theta=l^*$ the second term in
Eq.~(\ref{eq:Lnhat2}) converges to $\pm\infty$ and the DM would be
able to distinguish between them. On the other hand, if there is a
parameter value $l^*\in S$ such that the difference between the rates
$\mu_{l^*}^n$ and $\muzn$ satisfies
$|\mu_{l^*}^n - \muzn|<< \frac{1}{\sqrt{n}},$  then
under $\theta=l^*$ the second term in Eq.~(\ref{eq:Lnhat2}) converges
to $0$ and the DM would not be able to distinguish between them. The
analysis without Assumption \ref{asu:lambda_rates}.1 would be
similar, but with more complex notation.
\end{rem}

\section{The Limit of the Posterior Distribution Processes}\label{sec:Limit_Processes}

In this section we find the limit of the posterior distribution
processes, $\limn\ptn$, and study the relation between this limit and
the posterior distribution process $\pt$. In Section
\ref{sec:Densities_Assumptions} we formulate assumptions on the
density of the random variable $v$. In Section \ref{sec:Examples} we
provide examples of densities that satisfy these assumptions and an
example of a density that does not. Section
\ref{sec:The_Posterior_Process} gives the main theorems in the
paper. We find the limit of the posterior distribution processes and
discuss its properties. In Section \ref{sec:extensions} we discuss about
some generalizations.

\subsection{Assumptions on Densities}\label{sec:Densities_Assumptions}

In order to find the diffusion approximation for the sequence of
processes $\{\ptn\}_{n\inN}$, we need several assumptions on the
distribution of the random variable $v$. If no such assumptions are
made, then it may happen that for some $l\in S$ and some $t\geq 0 $,
the posterior probability $\ptn(l,t)$ will vanish with a positive probability
for some $n\inN$. Such cases differ from each other in the form of
their analysis and require different tools than the ones that we are
using in this paper. The following assumption states that the support of the interarrival times is the positive part of the axis. This assumption rules out a situation
where a single arrival can reveal a lot of information.
\begin{assumption}\label{asu:density1}
The random variable $v$ has the probability density
function\footnote{$\mathcal{C}^3$ is the class of real-valued
functions with continuous third derivative.} $({\rm pdf})$ $f \in
\mathcal{C}^3$, with the support $(0,\infty)$.
\end{assumption}

\begin{rem}\label{rem:density_f_to_f_theta}
For every $l\in S$, denote by $F_l^n$ the cumulative distribution
function (cdf) of $v_{1,l}^n$ and by $f_l^n$ the pdf of $v_{1,l}^n$.
By Assumption \ref{asu:t_u_t_v}, for every $s> 0$, $F_l^n (s) =
F(s\muln)$, where $F(s)$ is the cdf of $v$. Moreover, for every $s>
0$, $f_l^n (s) = \muln f (s\muln)$, while $f_l^n (s) =0$ for $s\leq
0$. In particular, for $l\in S$ the support of $v_{1,l}^n$ is
$(0,\infty)$.
\end{rem}

In the analysis of the posterior distribution process $(\ptn(l,t))$,
the following log-likelihood terms will appear:
$$\ln\left(\frac{f_l(s)}{f_0(s)}\right),\;
\ln\left(\frac{1-F_l(s)}{1-F_0(s)}\right).
$$
When we will use the representations of the cdfs that were
introduced in Remark \ref{rem:density_f_to_f_theta} and the Taylor
approximation for the log-likelihood ratios above, we will encounter the terms
$$\frac{f'(s)}{f(s)},
\frac{f(s)}{1-F(s)}.
$$
The following assumptions state that these functions are
``sufficiently" bounded.
%
%
%
\begin{assumption}\label{asu:density22}
$\\ $\ref{asu:density22}.1. The random variable $\frac{f'(v)}{f(v)}v$
has a finite standard deviation, denoted by $\sigmaf$.
%
%
$\\ $\ref{asu:density22}.2. There exist a monotone nondecreasing
function $M(x)$ and a positive parameter $\epsilon_M>0$, such that
for every $x\in(0,\infty)$
$$\left|\left(\frac{f'(x)}{f(x)}\right)''x^3\right|\leq M(x)$$
and
$$E\left[M\left(\left(1+\epsilon_M\right)v\right)\right] < \infty.$$
$\\ $\ref{asu:density22}.3. There exist a monotone nondecreasing
function $N(x)$ and a positive parameter $\epsilon_N>0$, such that
for every $x\in(0,\infty)$
$$\left|\frac{xf(x)}{1-F(x)}\right|\leq N(x)$$
and
$$E\left[N\left(\left(1+\epsilon_N\right)v\right)\right]^2<\infty.$$
\end{assumption}

\subsection{Examples}\label{sec:Examples}
In this section we provide an example of a family of distributions that
satisfy Assumptions \ref{asu:expectation1}, \ref{asu:density1}, and
\ref{asu:density22}. In fact, most of the frequently used continuous
distributions that satisfy Assumptions \ref{asu:expectation1} and
\ref{asu:density1} also satisfy Assumption \ref{asu:density22}. We
then present an example of a distribution that satisfies Assumptions
\ref{asu:expectation1} and \ref{asu:density1}, but not Assumption
\ref{asu:density22}.1. This example illustrates that Assumption
\ref{asu:density22}.1 is independent of Assumptions
\ref{asu:expectation1} and \ref{asu:density1}. As for Assumptions
\ref{asu:density22}.2 and \ref{asu:density22}.3, we do not know
whether they follow from previous assumptions or they are
independent of them.

\subsubsection{An Example that satisfies the Assumptions}\label{sec:Positive_Example}

Let $v$ be a random variable with the support $(0,\infty)$ whose pdf
is given by
$$f(x) = w_1(x)e^{w_2(x)},$$
where $w_1(x)$ and $w_2(x)$ are sums of power functions and the
powers in $w_2(x)$ are positive. Denote by $d$ the highest power in
$w_2(x)$. Since $f(x)$ is a pdf with the support $(0,\infty)$,
$w_1(x)>0$ and the smallest power in $w_1(x)$ is higher than $-1$.
Clearly $f \in \mathcal{C}^3$. By simple computations one can verify that
there exists a constant $C$, such that for every $x\in(0,\infty)$
the following holds
\begin{align}\notag
\left|\frac{f'(x)}{f(x)}x\right|,
\left|\left(\frac{f'(x)}{f(x)}\right)''x^3\right|,
\left|\frac{xf(x)}{1-F(x)}\right|\leq C\max\{1,x^d\}.
\end{align}
Assumption \ref{asu:density22}.1 follows since
$E[\max\{1,x^d\}]^2<\infty$. Since for every $\epsilon>0$,
$$E[\max\{1,((1+\epsilon)x)^d\}], E[\max\{1,((1+\epsilon)x)^d\}]^2<\infty,$$
and the functions $C\max\{1,x^d\}$ and $(C\max\{1,x^d\})^2$ are
monotone nondecreasing, it follows that Assumptions
\ref{asu:density22}.2 and \ref{asu:density22}.3 hold by choosing
$M(x):=C\max\{1,x^d\}$, $N(x):=C^2\max\{1,x^d\}^2,$ and $\epsilon_N,
\epsilon_M$ to be arbitrary positive constants.

This family of distributions contains the gamma, Weibull,
Maxwell--Boltzmann, and Rayleigh distributions.

\subsubsection{An Example that does not Satisfy Assumption \ref{asu:density22}.1}\label{sec:Counterexample} We show that
there exists a random variable that satisfies Assumptions
\ref{asu:expectation1} and \ref{asu:density1} and fails to satisfy
Assumption \ref{asu:density22}.1. For every $d\in\mathbb{N}$, define
$x_d:=1+\frac{1}{2}+\cdots+\frac{1}{d}$. Let $g(x)$ be the following
function: for every $d\in\mathbb{N}$ and every $x_d < x \leq
x_{d+1}$ define $g(x):=\frac{1}{d+1}(x-x_d)$. Then $g$ is not a pdf
of a random variable and it is not differentiable. However, by
smoothing $g$ and changing its values on a bounded interval, one can
construct a random variable with a pdf function $f$ that satisfies
Assumptions \ref{asu:expectation1} and \ref{asu:density1}. Since
Assumption \ref{asu:expectation1}.1 follows from Assumption
\ref{asu:expectation1}.2 by normalization, and Assumption
\ref{asu:density1} is only a matter of smoothing, it is sufficient to show that
Assumption \ref{asu:expectation1}.2 can be satisfied. This follows
from the following series of equalities and inequalities:
\begin{align}\notag
\int_0^\infty x^2g(x)dx &= \sum_{d=1}^{\infty}
\int_{x_{d}}^{x_{d+1}} x^2g(x)dx \leq \sum_{d=1}^{\infty}
x_{d+1}^2\int_{x_d}^{x_{d+1}} g(x)dx =
\frac{1}{2}\sum_{d=1}^{\infty} x_{d+1}^2\frac{1}{(d+1)^3}
\\\notag
&\approx  \frac{1}{2}\sum_{d=1}^{\infty}
\ln^2(d+1)\frac{1}{(d+1)^3}<\infty.
\end{align}
Assumption \ref{asu:density22}.1, however, does not hold, since
\begin{align}\notag
\int_0^\infty \left(\frac{g'(x)}{g(x)}x\right)^2g(x)dx &=
\sum_{d=1}^{\infty} \int_{x_d}^{x_{d+1}}
\left(\frac{g'(x)}{g(x)}x\right)^2g(x)dx \geq \sum_{d=1}^{\infty}
(d+1)^2x_d^2\int_{x_d}^{x_{d+1}} g(x)dx \\\notag
&=\frac{1}{2}\sum_{d=1}^{\infty} (d+1)^2x_d^2\frac{1}{(d+1)^3}
\approx \frac{1}{2}\sum_{d=1}^{\infty}
\ln^2(d+1)\frac{1}{d+1}=\infty.
\end{align}

\subsection{The Limit of the Posterior Distribution Processes}\label{sec:The_Posterior_Process}
Similar to the constructions of the Radon--Nikod\'{y}m process and
the posterior distribution process in the model of the Brownian motion with an
unknown drift (see Eqs.~(\ref{varphi}) and (\ref{frac_q_l_q_m})),
one can show that for every $l\in S$ and every $t\in[0,\infty)$,
\begin{align}\label{frac_p_l_p_m}
\ptn(l,t) = \frac{\pi_l\vphitn (l,t)}{\sum_{k\in S} \pi_k\vphitn
(k,t)}, \;\;l\in S,\;t\in[0,\infty),
\end{align}
where \begin{align}\label{varphin} \vphitn(l,t):=
\frac{d(P^n_l\mid\mathcal{F}^{L^n}_{nt} )
}{d(P^n_0\mid\mathcal{F}^{L^n}_{nt} )}, \;\;l\in S,\;t\in[0,\infty)
\end{align}
is the Radon--Nikod\'{y}m process. In fact, for every $l\in S$,
$(\vphitn (l,t))$ is the likelihood ratio process w.r.t.~$t$ and it
satisfies
\begin{align}\label{varphin2}
\vphitn (l,t) = \frac{P^n(L^n(nt), v^n_1, \ldots, v^n_{L^n(nt)}; nt
\mid \theta=l, t^n_v )} {P^n(L^n(nt), v^n_1, \ldots, v^n_{L^n(nt)};
nt \mid \theta=0, t^n_v )}, \;\;t\in[0,\infty).
\end{align}
The next theorem shows that the Radon--Nikod\'{y}m process and the
posterior distribution process of the $n$-th system converge to
$\vphis$ and $\pis$, respectively, for a properly chosen $\sigma$.

\begin{thm}[Main Theorem]\label{thm:convergence_ptt}
Under Assumptions \ref{asu:t_u_t_v}, \ref{asu:expectation1},
\ref{asu:lambda_rates}, \ref{asu:density1}, and \ref{asu:density22},
the following hold:
\begin{align}\label{Main_varphi}
\limn\vphitn&\eqd \vphiasf,\\\label{Main_p}
\limn\ptn&\eqd\piasf.
\end{align}
\end{thm}

\begin{rem}
The quantity $\sigmaf$ in Eqs.~(\ref{Main_varphi}) and
(\ref{Main_p}) depends on the pdf of the random variable $v$. That
is, the limit of the posterior distribution process depends on the
structure of the density of $v$ and not only on its moments.
Notice also that $\frac{\sqrt{\alpha}}{\sigmaf}$ is not the parameter that is associated with $(\tilde{L}(t))$ (see Eq.~(\ref{eq:hat_L})).
The relation between the parameters $\sigma_v$ and $\frac{1}{\sigmaf}$ is studied in Theorems \ref{thm:sf} and \ref{thm:difference} below.
\end{rem}

The proof of Theorem \ref{thm:convergence_ptt} is given in the
Appendix. We now outline the main ideas of the proof. We start with
the first part of the theorem. First, we show that (a) the fact that
the system was activated before time $t=0$, and (b) the
lack of arrivals during the time interval
$\left(\sum_{i=1}^{L^n(nt)}\vin,t\right]$, have almost-surely an
effect of order $o(1)$ on the posterior distribution process as $n$ goes
to infinity (Lemma \ref{lemma0}). Therefore, there is no significant
difference if the DM updates his belief only at arrival times. That
is,\footnote{The a.s. convergence is with respect to the metric
$e_\infty$.}
\begin{align}\label{eq:idea_proof}
\vphitn(l,t)= \exp\left\{\sum_{i=1}^{L^n(nt)}\ln\left(\frac{
f^n_l(\vin)}{f^n_0(\vin)}\right) +o(1)\right\} \;\; {\rm a.s.},
\end{align}
where $nt = \sum_{i=1}^{L^n(nt)}\vin$ (this means that time $nt$ is
an arrival time), see Eqs.~(\ref{eq:ptn_devided_by_1-ptn2}), (\ref{limAn0_part_1b}), and (\ref{limAn0_part_2b}). Second,
we find the distribution of $\sum_{i=1}^{L^n(nt)}\ln\left(\frac{
f^n_l(\vin)}{f^n_0(\vin)}\right)$.
We show that for every $n\inN$ there exists a process
$(\tilde{W}^n(t))$ with the following properties: it is independent
of $\theta$; the limit $\tilde{W}:\eqd\limn \tilde{W}^n$ exists and
the process $(\tilde{W} (t))$ is a standard Brownian motion such
that
\begin{align}\notag
\sum_{i=1}^{L^n(nt)}\ln\left(\frac{
f^n_l(\vin)}{f^n_0(\vin)}\right)=&\sqrt{n}\sigmaf\frac{\muln-\muzn}{\mukn}
\tilde{W}^n\left(\frac{L^n(nt)}{n}\right)\\\notag &-
\frac{1}{2}\left(\frac{l\sigmaf}{\alpha}\right)^2 \frac{L^n(nt)}{n}
+\left(\frac{l \theta\sigmaf^2}{\alpha^2}\right)\frac{L^n(nt)}{n}
+o(1)\;\;{\rm a.s.},
\end{align}
see Eqs.~(\ref{eq:lnvarphi_U3b}) and (\ref{eq:lemma1}). By taking the limit $n\rightarrow\infty$ and
using the random time-change theorem (see Chen and Yao (2001,
Theorem 5.3) \cite{Chen2001}) for the composition
$\tilde{W}^n\left(\frac{L^n(nt)}{n}\right)$ one gets the desired
result (see Proposition \ref{prop2}).\footnote{see
Eq.~(\ref{varphi_Z}) for the structure of $\vphis$.}

We now turn to the second part of the theorem. Notice that if we
prove Eq.~(\ref{Main_varphi}), then Eq.~(\ref{Main_p}) follows from
the definitions of $\vphis $ and $\pis$ and from
Eqs.~(\ref{frac_q_l_q_m}) and (\ref{frac_p_l_p_m}), because the
mapping $\vphi(\cdot,\cdot)\mapsto
\frac{\pi_l\vphi(\cdot,\cdot)}{\sum_{k\in S}
 \pi_k\vphi(k,\cdot)}$ is continuous.
%
%
%
%

From Eq.~(\ref{Main_p}) and the definition of $\pis $, it follows
that $\limn\ptn$ is distributed as a posterior distribution process
of a Brownian motion with an unknown drift. The following theorem
summarizes this observation.
\begin{thm}\label{thm:sf}
Under Assumptions \ref{asu:t_u_t_v}, \ref{asu:expectation1},
\ref{asu:lambda_rates}, \ref{asu:density1}, and \ref{asu:density22},
the process $\limn\ptn$ can be expressed as the posterior
distribution process of the process
\begin{align}\label{eq:hat_L2}
\hat{M}(t)=\hat{M}_\theta(t):=\theta t + \sqrt{\alpha}\frac{1}{\sigmaf} W'(t),
\;\;t\in[0,\infty),
\end{align}
where $(W'(t))$ is a standard Brownian motion independent of
$\theta$. Moreover, $\frac{1}{\sigmaf}\leq \sigma_v$ where equality
holds if and only if the random variable $v$ has a gamma
distribution (with expectation $1$).
\end{thm}
Since $\frac{1}{\sigmaf}\leq \sigma_v$, the paths of the process
$(\hat{M}(t))$ will be more concentrated around the path of the linear drift, $(\theta t)$, than the paths of the process $(\tilde{L}(t))$. In other words,
the process $(\hat{M}(t))$ is less noisy than $(\tilde{L}(t))$.
Therefore, it is easier to estimate the parameter $\theta$ given
$(\hat{M}(t))$ than given $(\tilde{L}(t))$. That is, $\limn\ptn$ is
more informative than $\pt$.
%
%
%
%
\begin{rem}
If $v$ has a gamma distribution with expectation $1$, then its
density is of the form $f(s) =
\frac{\beta^\beta}{\Gamma(\beta)}s^{\beta-1}e^{-\beta s}$, where
$\beta$ is a positive constant. From Remark
\ref{rem:density_f_to_f_theta} and Eq.~(\ref{eq:idea_proof}) it
follows that
\begin{align}
\ln(\vphit^n)(l,s)= L^n(ns)\ln\left(\frac{\muln}{\muzn}\right)^\beta
-ns\beta(\muln -\muzn) + o(1).
\end{align}
That is, for sufficiently large $n\inN$, the Radon--Nikod\'{y}m
density, and therefore also the posterior distribution process at
time $nt$, depend on the process $(L^n(ns))_{0\leq s\leq t}$ only
through $L^n(nt)$, up to order $o(1)$. Loosely speaking, for sufficiently large $n$'s the parameter $\theta$ has sufficient
statistics (based on $(\tilde{L}^n(s))_{s\leq t}$) that are `approximately independent of the past'. This is the same property
that holds in the Brownian motion with an unknown drift model
(see Remark \ref{rem:brown_present}). Therefore, we expect that indeed
this case the processes $\limn\ptn$ and $\pt$ will be
identically distributed, because no information is lost by looking
at the present rather than at the past.
\end{rem}

Before proving Theorem \ref{thm:sf}, we state a lemma that provides
insights about the parameter $\sigmaf$, which is then used in the
proof.
%
%

\begin{lem}\label{lem:R}
Under Assumptions \ref{asu:expectation1}, \ref{asu:density1}, and
\ref{asu:density22}.1, the following equalities hold:
\begin{align}\label{eq:R1}
&E\left[\frac{f'(v)}{f(v)}v\right] = -1
\end{align}
and
\begin{align}\label{eq:R2}
E\left[\left(\frac{f'(v)}{f(v)}\right)' v^2\right] = 1- \sigmaf^2.
\end{align}
\end{lem}
\begin{proof}

\begin{align}\notag
E\left[\frac{f'(v)}{f(v)}v\right] =\left.
\int_0^\infty{f'(v)v \, dv} = {f(v)v}\right|_0^\infty -
\int_0^\infty{f(v) dv} = -1,
\end{align}
where the last equality holds since
$$\int_0^\infty{f(v)v\, dv}, \int_0^1{f(v) dv}<\infty$$ and therefore
$$\limu f(v)v =\limuz f(v)v = 0.$$
From Assumption \ref{asu:density22}.1, and by using similar arguments
as above, it follows that
$$
E\left[\left(\frac{f'(v)}{f(v)}\right)' v^2\right] = 1- \sigmaf^2.
$$
\end{proof}

\begin{proof}[Proof of Theorem \ref{thm:sf}]
From the definition of $\pis$ it follows that the posterior
distribution process of the process $(\hat{M}(t))$ is given by
$\piasf$. From Eq.~(\ref{Main_p}) it follows that $\limn\ptn$ is
distributed as $\piasf$. We now show that
\begin{align}\notag
\sigmaf \sigma_v \geq 1,
\end{align}
and that equality holds if and only if $v$ has a gamma distribution with
expectation $1$. The inequality follows from the following
relations:
\begin{align}\notag
&\sigmaf \sigma_v = \sqrt{E\left[\frac{f'(v)}{f(v)}v+1\right]^2
E\left[v-1\right]^2}\geq\left|E\left[\left(\frac{f'(v)}{f(v)}v+1\right)\left(v-1\right)\right]\right|=1.
\end{align}
The first equality holds by the definitions of $\sigmaf$ and
$\sigma_v$, Assumption \ref{asu:expectation1}.1 and by Lemma \ref{lem:R} (Eq.~(\ref{eq:R1})). The
inequality is the Cauchy--Schwartz inequality. The second equality
follows from Lemma \ref{lem:R} (Eq.~(\ref{eq:R2})) and from the
equation
$$E\left[\frac{f'(v)}{f(v)}v^2\right]=-2,$$ which is obtained via integration by parts.
Notice that the inequality turns into equality if and only if
$\frac{f'(v)}{f(v)}v+1$ and $v-1$ are linearly dependent. One can verify
that under Assumptions \ref{asu:expectation1} and \ref{asu:density1}
this happens if and only if $v$ has a gamma distribution with
expectation $1$.
\end{proof}


The next theorem states that the difference between $\sigma_v$ and $\frac{1}{\sigmaf}$ can be arbitrarily large.
Hence, the distributions of $\piasf$ (and by Theorem \ref{thm:convergence_ptt} also $\ptn$) and $\piasv$ can be very different.
\begin{thm}\label{thm:difference}
The difference $\sigma_{v} - \frac{1}{\sigma_{f}}$ can be arbitrarily large.
\end{thm}
The proof of Theorem \ref{thm:difference} is given in the Appendix. To show that $\sigma_v - \frac{1}{\sigmaf}$ can be arbitrarily large we construct a family of random variables that satisfy Assumptions \ref{asu:expectation1}, \ref{asu:density1}, and \ref{asu:density22} and for which the variances $\sigma_v^2$'s can be arbitrarily large and the parameters $\frac{1}{\sigma_f^2}$'s are uniformly bounded from above.

In Sections \ref{sec:Asymptotic1} and \ref{sec:Asymptotic2} below we show how to use the distribution of $\piasf$ in order to solve optimal stopping problems w.r.t.~the observed process $(\tilde{L}^n(t))$. We show there that if one calculates his strategy based on the distribution of $\piasv$ instead of the distribution of $\piasf$, then his payoff will be suboptimal. By Theorem  \ref{thm:difference} it turns out that the strategies and the payoffs that follow by the distributions of $\piasv$ and $\piasf$ can be very different and therefore by taking the wrong approximation, the performance can be relatively bad (see Remark \ref{rem:distance_vf} below).

\subsection{Generalizations}\label{sec:extensions}
\subsubsection{Intermittent System} There are cases where the system
operates intermittently. For example, the departure process from a
G/G/1 queue with an unknown service rate can be modeled as the
system described above that operates only when the queue is not empty
(with `departures' instead of `arrivals'). In this section we study
systems that operate intermittently, and let $(B^n(t))$ be the
process that represents the cumulative time that the $n$-th system works
during the time interval $[0,t]$. Let
$$\boldsymbol{\tilde{\pi}^n_B}(l,t) :=  \boldsymbol{\tilde{\pi}^n}(l,B^n(t)) =  \boldsymbol{\pi^n}(l,B^n(nt)),\;\;l\in S\;,t\in[0,\infty),$$
be the posterior distribution process for the observed process
$(L^n(B^n(nt)))$. The following theorem describes the distribution
of $\limn\boldsymbol{\tilde{\pi}^n_B}$.
\begin{thm}\label{thm:inter}
Suppose that there is a constant $0\leq \rho\leq 1$ such that
$\limn\frac{B^n(nt)}{n} = \rho t$ $\rm{ u.o.c.}$ Under Assumptions
\ref{asu:t_u_t_v}, \ref{asu:expectation1}, \ref{asu:lambda_rates},
\ref{asu:density1}, and \ref{asu:density22}, the following holds:
\begin{align}
&\limn\boldsymbol{\tilde{\pi}^n_B}
\eqd\boldsymbol{\pi}_{\frac{\sqrt{\alpha}}{\sigmaf\sqrt{\rho}}}.
\end{align}
\end{thm}
The proof follows from the random time-change theorem (Chen and Yao
(2001, Theorem 5.3) \cite{Chen2001}) in a similar way to the proof
of Theorem \ref{thm:convergence_ptt}, and is therefore omitted.

\subsubsection{Continuous Distribution over
$\theta$}\label{sec:continuum} Theorems \ref{thm:convergence_ptt},
\ref{thm:sf}, and \ref{thm:inter} also hold in case that
$\theta$ is a continuous random variable with the density $\pi_l$,
$l\in S$. In this case, the term $\sum_{k\in S}$ in
Eqs.~(\ref{frac_q_l_q_m}) and (\ref{frac_p_l_p_m}) is replaced by
$\int_{k\in S}$.

\section{Optimal Stopping Problems}\label{sec:optimal_stopping_time}
The problem of finding closed-form solutions for optimal stopping problems w.r.t.~$(L^n(t))$ in the general case suffers from high complexity.  Buonaguidi and Muliere (2013) \cite{Buonaguidi2013} and Cohen and Solan (2013) \cite{Cohen2013} solved such optimal stopping problems in case that, given $\theta$, the process $(L^n(t))$ is a L\'{e}vy process. We do not make that assumption and rather find an asymptotically optimal solution by using the limit process $\limn\ptn$. As mentioned in Section \ref{sec:introduction}, there are several
optimal stopping problems that have been studied in the literature
with respect to a Brownian motion with an unknown drift. The purpose
of this section is to show that optimal stopping problems such as
the Bayesian Brownian bandit problem (Berry and Friestedt (1985) \cite{Berry1985},
Bolton and Harris (1999) \cite{Bolton1999}, Cohen and Solan (2013) \cite{Cohen2013}) and the sequential testing problem\footnote{We
consider here discounted optimal stopping problems, whereas Shiryaev
considers an undiscounted problem.} (Shiryaev (1978)
\cite{Shiryaev1978}), are relevant for a process that is close in
distribution to a Brownian motion with an unknown drift. These
papers considered a Brownian motion with an unknown drift where
there are only two hypotheses about the drift, and therefore we
limit the discussion on this section to the case of two available hypotheses $H_l$
and $H_0$, where $0\neq l\in\mathbb{R}$. The optimal stopping
problems consist of (a) an observed process, (b) a stopping time
adapted to the observed process, and (c) a payoff function that is a
function of the observed process. Although the optimal stopping
problems are formulated with the observed process, which is a
Brownian motion with an unknown drift, it is possible to formulate
the problems and their solutions in terms of the posterior
distribution process. We present a sequence of random parameter
systems that converges to a Brownian motion with an unknown drift.
Under modest assumptions we formulate a stopping time problem with
respect to the posterior distribution process
$(\ptn(l,t),\ptn(0,t))$. We solve these problems by using Theorem
\ref{thm:convergence_ptt}, and we deduce from Theorem \ref{thm:difference} that by using the approximation $\pt$ instead of $\limn\ptn$,
the performance can be relatively bad.

In Section \ref{sec:Cost_Function} we define the cost function and
the optimal stopping problems with respect to the posterior
distribution process $(\ptn(l,t),\ptn(0,t))$. In Section
\ref{sec:Stoppin_Times} we find an approximate solution by using
Theorem \ref{thm:convergence_ptt}. In Sections \ref{sec:bandit} and
\ref{sec:sequential} we show that the Bayesian Brownian bandit problem and
the Brownian sequential testing problem are special cases of the
general problem that is described here.

Define $\vphihat:\eqd\limn\vphitn$ and $\phat:\eqd\limn\ptn$. From
Eq.~(\ref{Main_varphi}) it follows that $\vphihat$ is distributed as
$\vphiasf$.


Recall that in this section we study the case where the support of $\theta$ consists of two states: $0$ and $l$. By knowing the prior/posterior probability of one state, the DM can infer the probability of the other. Therefore, it is sufficient to make the forthcoming analysis w.r.t.~the following processes $\phat(t):=\phat(l,t)$,
$\ptn(t):=\ptn(l,t)$, $\vphihat(t):=\vphihat(l,t)$,
$\vphitn(t):=\vphitn(l,t)$, $t\in[0,\infty)$, and the prior probability $\pi:=\pi_l$.


\subsection{The Cost Function}\label{sec:Cost_Function}
Suppose that a DM who operates the $n$-th system, observes the process
$(L^n(t))$, and continuously updates his belief about the hypotheses
$H_l$ and $H_0$. Let $k^n, K^n:[0,1]\rightarrow \mathbb{R}$ be two
functions that stand for the instantaneous cost and for the terminal
cost, respectively; the DM's instantaneous discounted
cost\footnote{Notice that the discount factor $r$ is scaled by an
order of $n$.} for operating the system during the time interval
$[t,t+dt)$ is $\frac{r}{n}e^{-\frac{r}{n}t}k^n(\pn(t))dt$, where
$\pn(t):=P(\theta = l \mid \mathcal{F}_{t}^{L^n} ; \pi)$. The choice
that the DM should make is when to stop operating the system. If the
DM stops at time $T$  then he has an additional discounted cost of
$\frac{r}{n}e^{-\frac{r}{n}T}K^n(\pn(T))$. Formally, the DM chooses
a stopping time $\tau^n$ for the process $(L^n(t))$; that is, the
stopping time is adapted to the filtration $\mathcal{F}_t^{L^n}$,
which is the natural filtration generated by $(L^n(t))$. The
expected discounted loss of the DM if he chooses the stopping time
$\tau^n$ is
\begin{align}\label{risk_system_general}
&V^n_{\tau^n} (\pi):=E^\pi\left[\int_0^{\tau^n}
{\frac{r}{n}e^{-\frac{r}{n}t}k^n(\pn(t))dt}+\frac{r}{n}e^{-\frac{r}{n}\tau^n}K^n(\pn(\tau^n))
\right].
\end{align}
Set $\tautn:= \tfrac{1}{n}\tau^n $. The stopping time $\tautn$ is
adapted to the filtration $\mathcal{F}_{nt}^{L^n}$, which is identical
to the filtration $\mathcal{F}_t^{\ptn}$.
Eq.~(\ref{risk_system_general}) is equivalent to\footnote{Recall
that for every $t>0$ we defined $\ptn(t):=\p^n(nt)$ (see
Eq.~(\ref{eq:tilde_pi})).}
\begin{align}\label{risk_system_general_2}
&V^n_{\tautn}(\pi)=E^\pi\left[\int_0^{\tautn}
{re^{-rt}k^n(\ptn(t))dt}
+ \frac{r}{n}e^{-r\tautn}K^n(\ptn(\tautn)) \right].
\end{align}
The goal of the DM is to minimize $V^n_{\tautn} (p)$ and to find, if
exists, the optimal stopping time $\tau^{*,n}$ for which the infimum
of (\ref{risk_system_general_2}) is attained. Let
\begin{align}\label{value_general}
&U^n(\pi):=\underset{\tautn}{\inf}\,V^n_{\tautn} (\pi)
\end{align}
be the minimal loss that the DM can achieve, and in case that the infimum
is attained, let
\begin{align}\label{tau*_general}
&\taut^{*,n} (\pi) \in
\underset{\tautn}{\text{arg}\min}\,V^n_{\tautn} (\pi)
\end{align}
be an optimal stopping time given that the prior belief is $\pi$.

\begin{assumption}\label{asu:value_functions}
$\\ $\ref{asu:value_functions}.1. The sequence of functions $k^n$
converges uniformly to a function $k$ on $[0,1]$.
$\\ $\ref{asu:value_functions}.2. $k$ is continuous on the interval
$[0,1]$.
$\\ $\ref{asu:value_functions}.3. The sequence of functions $K^n /n$
converges uniformly to a function $K$ on $[0,1]$.
$\\ $\ref{asu:value_functions}.4. $K$ is continuous on the interval
$[0,1]$.
\end{assumption}

\begin{rem}\label{rem:k_K_bounded}
From Assumption \ref{asu:value_functions}.2 (resp.
\ref{asu:value_functions}.4) it follows that the function $k$ (resp.
$K$) is bounded and uniformly continuous on $[0,1]$. From Assumption
\ref{asu:value_functions}.1 (resp. \ref{asu:value_functions}.3) it
follows that there exists a constant $C_k>0$ (resp. $C_K$), such
that for every $n\inN$ and every $\pi\in[0,1]$, one has $|k^n(\pi)|,
|k(\pi)|\leq C_k$ (resp. $|K^n(\pi)/n|, |K(\pi)|\leq C_K$).
\end{rem}

We now define the expected cost and the value function with respect
to $\mathcal{F}^{\phat}_t$. Fix $\pi\in[0,1]$. Then the expected cost
by using the $\mathcal{F}^{\phat}_t$-adapted stopping time $\tau$ is
\begin{align}\notag
&V_\tau(\pi):=E^\pi\left[\int_0^{\tau} {re^{-rt}k(\phat(t))dt}
+ re^{-r\tau}K(\phat(\tau)) \right].
\end{align}
Let
\begin{align}\label{value_general_limit}
&U(\pi):=\underset{\tau}{\inf}\,V_{\tau} (\pi)
\end{align}
be the value function, and in case that the infimum is attained, let
\begin{align}\label{tau*_general_limit}
&\tau^{*} (\pi) \in \underset{\tau}{\text{arg}\min}\,V_{\tau} (\pi)
\end{align}
be an optimal stopping time given that the prior belief is $\pi$.

\subsection{Stopping Times}\label{sec:Stoppin_Times}
Since the optimal stopping times (if exist) of the problems
(\ref{value_general})--(\ref{tau*_general}) and
(\ref{value_general_limit})--(\ref{tau*_general_limit}) are
stationary Markovian stopping times with respect to the posterior
distributions processes $(\ptn(t))$ and $(\phat(t))$, respectively
(see Cohen and Solan (2013, Remark 4) \cite{Cohen2013}), it is
natural to confine our discussion to the set of stationary Markovian
stopping times. We now define a \emph{first exit time strategy}. To
this end, we define a subset of $[0,1]$ such that if the posterior
is within this subset, then the DM continues and stops otherwise. Let
$D=\bigcup_{i}(a_i, b_i)\subseteq [0,1]$ be a finite union of
disjoint open intervals such that if $b_j=1$ (resp. $a_i=0$), then
the open interval $(a_j,b_j)$ (resp. $(a_i,b_i)$) is replaced by the
semi-open interval $(a_j,1]$ (resp. $[0,b_i)$).
\begin{assumption}\label{asu:D}
For every $i< j$ one has $b_i < a_j$.
\end{assumption}
Assumption \ref{asu:D} merely says that the intervals do not `touch
each other'. Define\footnote{The subscript $\pi$ indicates the prior
probability that $\theta=l$. That is,
$\boldsymbol{\tilde{\pi}}^{\boldsymbol{n}}_{\pi}(0) = \pi$ and
$\phat_{\pi}(0) = \pi$.}
\begin{align}\label{eq:tautnD}
\tautnD (\pi)&:=\inf\{t \mid
\boldsymbol{\tilde{\pi}}^{\boldsymbol{n}}_{\pi}(t)\notin
D\},\\\label{eq:tautD}
\tautD (\pi)&:=\inf\{t \mid \phat_{\pi}(t)\notin D\}.
\end{align}
That is, $D$ is the \emph{continuation region} with respect to the
stopping times $\{\tautnD\}_{n\inN}$ and $\tautD$. From Assumption
\ref{asu:D} it follows that if the DM continues for every prior in a
certain punctured neighborhood of $a$, then he should also continue for the prior $a$.

The next theorem asserts that by using the same continuation region
$D$ for every $n\inN$, the stopping times $\tautnD(\pi)$ converge in
distribution to $\tautD(\pi)$ and the expected cost functions
$V^n_{\tautnD}(\pi)$ converge to $V_{\tautD}(\pi)$.

\begin{thm}\label{thm:convergence_Vn}
Under Assumptions \ref{asu:t_u_t_v}, \ref{asu:expectation1},
\ref{asu:lambda_rates}, \ref{asu:density1}, \ref{asu:density22},
\ref{asu:value_functions}, and \ref{asu:D}, we have
\begin{align}\label{eq:convergence_tautnD_to tautD}
\limn \tautnD(\pi) \eqd \tautD(\pi)
\end{align}
and
\begin{align}\label{eq:convergence_Vtautn_to Vtaut_in ecpectation}
\limn V^n_{\tautnD}(\pi) = V_{\tautD}(\pi).
\end{align}
\end{thm}
The proof is relegated to the Appendix. In fact, Theorem \ref{thm:convergence_Vn} holds even if we replace the $D$'s on the left-hand sides of Eqs.~(\ref{eq:convergence_tautnD_to tautD}) and (\ref{eq:convergence_Vtautn_to Vtaut_in ecpectation}) by $D^n$'s, where $D^n\rightarrow D$ in the sense that the indicators of $D^n$ converge pointwise to the indicator of $D$. The proof requires some technical modifications that we wish to avoid in order to ease the notation.

In some models such as the Bayesian Brownian bandit and the Sequential testing (as shown in Sections \ref{sec:Asymptotic1} and \ref{sec:Asymptotic2} respectively)
the limit problem admits a unique optimal stopping time that is associated with a continuation region $D^*$. That is, $V_{\taut_{D^*}} = U$. Therefore, by Theorem \ref{thm:convergence_Vn} it follows that for every $\pi\in[0,1]$ one has
\begin{align}\notag
\limn V^n_{\tilde{\tau}^n_{D^*}}(\pi) = V_{\tilde{\tau}_{D^*}}(\pi) = U(\pi),
\end{align}
whereas for every $\bar{D}\neq D^*$ and every $\pi\in[0,1]$ one has
\begin{align}\label{eq:suboptimal}
\limn V^n_{\tilde{\tau}^n_{\bar{D}}}(\pi) = V_{\tilde{\tau}_{\bar{D}}}(\pi) \geq U(\pi).
\end{align}
This is summarized in the following corollary.
\begin{cor}\label{cor:convergenceU}
Under Assumptions \ref{asu:t_u_t_v}, \ref{asu:expectation1},
\ref{asu:lambda_rates}, \ref{asu:density1}, \ref{asu:density22},
\ref{asu:value_functions}, and \ref{asu:D}, if the limit problem admits an optimal stopping time that is associated with a continuation region $D^*$, then
\begin{align}\notag
\limn U^n(\pi) = \limn V^n_{\tilde{\tau}^n_{D^*}}(\pi) = V_{\tilde{\tau}_{D^*}}(\pi) = U(\pi).
\end{align}
\end{cor}

\begin{rem}
For every $n$ we defined the expected discounted loss in
Eq.~(\ref{risk_system_general}) by using the functions $k^n$ and $K^n$,
and found an equivalent representation in
Eq.~(\ref{risk_system_general_2}). By Assumption
\ref{asu:value_functions}, the functions $k^n$ and $K^n/n$ converge
uniformly to the functions $k$ and $K$, respectively. Therefore, it
would not make much difference if we defined
\begin{align}\notag
&V^n_{\tautn_D} (\pi):=E^\pi\left[ R(\ptn)\right],
\end{align}
and
\begin{align}\notag
&V_{\taut_D} (\pi):=E^\pi\left[ R(\phat)\right],
\end{align}
where\footnote{Notice that according to Eqs.~(\ref{eq:tautnD}) and
(\ref{eq:tautD}), $\tautn_D$ and $\tautn_D$ are functions of the
processes $\ptn$ and $\phat$ respectively.}
\begin{align}\notag
&R(\p):= \int_0^{\tau_D }
{re^{-rt}k(\p(t))dt}+re^{-r\tau_D}K(\p(\tau_D))
\end{align}
and
\begin{align}\notag
\tau_D &:=\tau_D(\pi)=\inf\{t \mid \p_\pi(t)\notin D\}.
\end{align}
That is, for every $n\inN$ one has $k^n\equiv k$ and $K^n/n\equiv K$.
In this case, one may try to use the convergence in distribution
$\limn\ptn\eqd\phat$ and conclude that
$
\limn E^\pi [R(\ptn) ] =E^\pi [R(\phat) ].$
However, the function $R$ is not continuous with respect to the
process $\p$, since it is possible to exhibit two processes $\p_1$
and $\p_2$ that are relatively close, but that the stopping times
$\tau_D (\p_1)$ and $\tau_D (\p_2)$ are relatively far from each
other, in which case the difference $|R(\p_1)-R(\p_2)|$ may be large.
Hence, the inference that $ \limn E^\pi [R(\ptn) ] =E^\pi [R(\phat) ]$ holds is
incorrect.

\end{rem}



\subsection{Bayesian Brownian Bandit Problem}\label{sec:bandit}
In Sections \ref{sec:Cost_Function} and \ref{sec:Stoppin_Times} we studied a family of optimal stopping problems w.r.t.~a sequence of discrete processes whose weak limit is a Brownian motion with an unknown drift. In this section we provide an example of an optimal stopping problem for which the limit problem is the Bayesian Brownian bandit problem (see Berry and Friestedt (1985) \cite{Berry1985},
Bolton and Harris (1999) \cite{Bolton1999}, Cohen and Solan (2013) \cite{Cohen2013}). We provide an asymptotically optimal solution by using Theorem \ref{thm:convergence_Vn} and Corollary \ref{cor:convergenceU}. We also infer that if one calculates his strategy based on the distribution of $\pt$ instead of the distribution of $\limn\ptn$, then his payoff will be suboptimal.

A DM operates a system in continuous time which can be
of two types, High ($H_l$) or Low ($H_0$). The DM observes the
process $(L^n(t))$ where $n\inN$ is fixed and updates his belief continuously about the
hypotheses $H_l$ and $H_0$. For each job arriving to the system, the
DM gets $1$ dollar.
In addition, he pays $c^n$ dollars per time unit for operating
the system. The choice that the DM should make is when to stop
operating the system.
Formally, the DM should choose a stopping time $\tau^n$ for the
process $(L^n(t))$; that is, the stopping time is adapted to the
filtration $\mathcal{F}_t^{L^n}$. The expected discounted loss of
the DM if he chooses the stopping time $\tau^n$ is
\begin{align}\label{risk_system_bandit}
&V^n_{\tau^n} (\pi):=\sqrt{n}E^\pi\left[\int_0^{\tau^n}
{\frac{r}{n}e^{-\frac{r}{n}t}d(c^n t - L^n (t))}\right].
\end{align}
The goal of the DM is to minimize $V_\tau^n (\pi)$, and to find, if it
exists, the optimal stopping time $\tau^{*,n}$ for which the infimum
of (\ref{risk_system_bandit}) is attained.

We now present the cost function by using $\ptn$. Since for every
$k\in\{0,l\}$ one has $E[c^nt - L^n(t)\mid \theta = k]= (c^n -
\mu_k^n) t$, we naturally assume that $\mu_0^n<c^n<\mu_l^n$. That
is, the arrival rate is higher (resp. lower)  in the High (resp.
Low) type than the cost per time unit for operating the system;
otherwise, the problem would be degenerate: if $\mu_0^n < \mu_l^n <
c^n$ the DM will stop operating the system at time $0$, while if
$c^n < \mu_0^n < \mu_l^n$ he will operate it indefinitely.
By standard arguments (see Cohen and Solan (2013, Lemma 4)
\cite{Cohen2013}), one can represent the function $V^n_{\tau^n}
(\pi)$ as follows:
\begin{align}\label{risk_system_bandit_2}
&V^n_{\tau^n} (\pi) = \sqrt{n}E^\pi\left[\int_0^{\tau^n}
{\frac{r}{n}e^{-\frac{r}{n}t}[(c^n - \mu_l^n)\pn(t)  + (c^n-
\mu_0^n)(1-\pn(t)) ] dt } \right],
\end{align}
which by Eq.~(\ref{risk_system_general_2}) equals
\begin{align}\label{risk_system_bandit_3}
&V^n_{\tautn} (\pi) = \sqrt{n}E^\pi\left[\int_0^{\tautn} {re^{-rt}[(c^n -
\mu_l^n)\ptn(t)  + (c^n- \mu_0^n)(1-\ptn(t)) ] dt } \right].
\end{align}
That is, the cost functions $k^n$ and $K^n$ of the $n$-th system can
be represented as follows:
\begin{align}\label{kn_bandit}
k^n(\pi)=\sqrt{n}(c^n-\muln)\pi+\sqrt{n}(c^n-\muzn)(1-\pi),\;\;\pi\in[0,1]
\end{align}
and
\begin{align}\label{Kn_bandit}
K^n(\pi)\equiv 0,\;\;\pi\in[0,1].
\end{align}
Suppose that for every $n\inN$, $\muzn<c^n<\muln$. Moreover, we need
the following assumption that states that the High type is better
than the Low type by an ``$\frac{1}{\sqrt{n}}$ order style".

\begin{assumption}\label{asu:HT}
$\\ $ $c_0:=\limn\sqrt{n}(c^n-\muzn)
> 0 > \limn\sqrt{n}(c^n-\muln)=:c_l$.
\end{assumption}
Assumption \ref{asu:HT} says that the scaled limit of the difference
between the operation cost and the arrival rate in the High (resp.
Low) type yields a negative (resp. positive) expected loss. Under
Assumption \ref{asu:HT} it follows that $k^n$ converges uniformly on
$[0,1]$ to
\begin{align}\label{k_bandit}
&k(\pi):=c_l \pi + c_0(1-\pi),\;\;\pi\in[0,1].
\end{align}

\subsubsection{Asymptotic Optimality}\label{sec:Asymptotic1}
In this section we define cut-off strategies by using the notion of first exit time strategies of the posterior processes from an interval of the form $(\bar{p},1]$.
We call $\bar{p}$ the cut-off point. We prove that the $n$-th system admits a unique
optimal stopping time and that
it is a cut-off strategy. We will therefore restrict the class of
stopping times to the class of cut-off strategies.
We also show that for every cut-off point $\bar{p}$, the first exit time of the process $(\boldsymbol{\tilde{\pi}}^{\boldsymbol{n}}(t))$
from the interval $(\bar{p},1]$ and the payoff that is associated with this strategy, converge to the first exit time of the process $(\boldsymbol{\hat{\pi}}(t))$
from that interval $(\bar{p},1]$ and the payoff that is associated with this strategy, respectively. We conclude this section by finding asymptotically optimal stopping time and the asymptotic value function.

We start with a few properties of the value function $U^n(\pi)$ and
deduce that the optimal strategy in the $n$-th system is a cut-off
strategy. The proof is similar to the proof of Proposition 2 in
Cohen and Solan (2013) \cite{Cohen2013} and is therefore omitted.
\begin{prop}\label{prop:concave}
For every fixed $n\inN$, the function $\pi\mapsto U^n(\pi)$ is
monotone, nonincreasing, bounded from above by $0$, concave, and
continuous.
\end{prop}

\begin{rem}\label{rem:exit_time_1}
From Proposition \ref{prop:concave} it follows that there is a
cut-off point $p^{*,n}$ in $(0,1]$, such that $U^n(\pi)=0$ if
$\pi\leq p^{*,n}$, and $U^n(\pi)<0$ otherwise. That is, the
optimal strategy is to continue while the posterior lies in the
interval $(p^{*,n},1]$, and to stop otherwise. We call this
strategy a cut-off strategy with cut-off point $p^{*,n}$.
\end{rem}

Berry and Friestedt (1985) \cite{Berry1985} showed that the Bayesian Brownian bandit problem
admits a unique optimal strategy and that it is a cut-off strategy
w.r.t.~the posterior process of the Brownian motion with the unknown drift.
Denote by $p^*$ the cut-off point that is associated with the optimal cut-off w.r.t.~the limit process $\limn\ptn = \boldsymbol{\hat{\pi}}$ (which is distributed as $\piasf$).
That is, for every $\pi\in[0,1]$ one has $U(\pi) = V_{\tilde{\tau}_{(p^*,1]}}(\pi)$.
Recall that $(p^*,1]$ is the continuation region for the posterior process.
For every $n\inN$ and every $\bar{p}\in[0,1]$ define the
continuation region $(\bar{p},1]$. The next result follows from
Eqs.~(\ref{kn_bandit})--(\ref{k_bandit}), Theorem
\ref{thm:convergence_Vn}, and Corollary \ref{cor:convergenceU}.
\begin{thm}\label{thm:convergence_Vn1}
Fix $0\leq \bar{p} \leq 1$. Under Assumptions \ref{asu:t_u_t_v},
\ref{asu:expectation1}, \ref{asu:lambda_rates},
 \ref{asu:density1}, \ref{asu:density22}, and
\ref{asu:HT}, we have\footnote{The function $V_{\taut_{(\bar{p},1]}}(\pi)$ can be
expressed explicitly through the parameters of the problem, but
since it has no fundamental contribution, this expression is
omitted (see Berry and Friestedt (1985, pp.~171--172) \cite{Berry1985}).}
\begin{align}\label{eq:convergence_tau}
\limn \tautn_{(\bar{p},1]}(\pi) \eqd \taut_{(\bar{p},1]}(\pi),
\end{align}
%
%
\begin{align}\label{eq:convergence_Vn}
\limn V^n_{\tautn_{(\bar{p},1]}}(\pi) =
V_{\taut_{(\bar{p},1]}}(\pi)
\end{align}
and there exists $p^*\in[0,1]$ such that
\begin{align}\label{eq:Un}
\limn U^n(\pi) = \limn V^n_{\tilde{\tau}^n_{(p^*,1]}}(\pi) = V_{\tilde{\tau}_{(p^*,1]}}(\pi) = U(\pi).
\end{align}
%
%
%
%
\end{thm}

%
%
%

\begin{rem}\label{rem:distance_vf}
 From Eq.~(\ref{eq:Un}) it follows that in order to find
the asymptotically optimal cut-off point $p^*$, the DM must use the cut-off point
taken from the optimal solution of the Bayesian Brownian bandit problem w.r.t.~the posterior
process $\piasf$ and not w.r.t.~the posterior process $\piasv$. Denote by $p^*_f$ and $p^*_v$
the cut-off points that are associated with the Bayesian Brownian bandit problem w.r.t.~the posteriors $\piasf$ and $\piasv$, respectively.
Theorem \ref{thm:difference} states that the difference between $\sigma_v$ and $\frac{1}{\sigmaf}$
can be arbitrarily large and therefore the distributions of $\piasf$ and $\piasv$ can be relatively different,
and so the difference between the optimal cut-off points $p^*_f$ and $p^*_v$ can be arbitrarily large within the interval $[0,1]$.
By Eq.~(\ref{eq:convergence_Vn}) it follows that for every prior $\pi\in[0,1]$ and for sufficiently large $n$, the payoff that is associated with the cut-off point $p^*_v$ is approximately $V_{\tilde{\tau}_{(p^*_v,1]}}(\pi)$, which, by Eq.~(\ref{eq:suboptimal}), is greater than $V_{\tilde{\tau}_{(p^*_f,1]}}(\pi)=U(\pi)$. The difference between these functions can be relatively large, see Berry and Friestedt (1985, pp.~171--172) \cite{Berry1985} for closed-form formulas.
\end{rem}

\subsection{Discounted Sequential Testing}\label{sec:sequential}
In this section we provide an example of an optimal stopping problem w.r.t.~a sequence of discrete processes for which the limit problem is a discounted version of the sequential testing problem (Shiryaev (1978)
\cite{Shiryaev1978}). We provide an asymptotically optimal solution by using Theorem \ref{thm:convergence_Vn} and Corollary \ref{cor:convergenceU}. We also infer that if one calculates his strategy based on the distribution of $\pt$ instead of the distribution of $\limn\ptn$, then his payoff will be suboptimal.

Fix $n\inN$. The DM observes the process $(L^n(t))$ and continuously
updates his belief on the hypotheses $H_l$ and $H_0$. Using the
belief process, his goal is to test sequentially these hypotheses
with minimal loss. The choice that the DM should make is when to
stop operating the system, and at that time to guess which one of the two hypotheses holds.
Formally, the DM should choose a decision rule
$(\tau^n,d^n)$ for $(L^n(t))$,
that is, a stopping time $\tau^n$ that is adapted to the filtration
$\mathcal{F}_t^{L^n}$, and a decision function $d^n$ that is a
$\mathcal{F}_\tau^{L^n}$-measurable random variable taking the
values $0$ and $l$. The choice $d^n=l$ is interpreted to mean that
the DM accepts $H_l$, while the choice $d^n=0$ is interpreted to
mean that the DM accepts $H_0$. The expected loss of the DM under
the decision rule $(\tau^n,d^n)$ is
\begin{align}\label{risk_system_sequential_testing}
&Y^n_{(\tau^n,d^n)}(\pi):=E^\pi\left[\int_0^{\tau^n}
\frac{r}{n}e^{-\frac{r}{n}t}c^n dt +
\frac{r}{n}e^{-\frac{r}{n}{\tau^n}}(a^n \mathbb{I}_{(d^n=0,
\theta=l)} + b^n \mathbb{I}_{(d^n=l, \theta = 0)}) \right],
\end{align}
where $a^n$, $b^n$, and $c^n$ are given positive constants that
represent the cost of type ${\rm II}$ error, the cost of type ${\rm
I }$ error, and the operation cost per unit of time, respectively.
The goal of the DM is to minimize $Y^n_{(\tau^n,d^n)}(\pi)$, and to
find, if exists, the optimal stopping rule $(\tau^{*,n},d^{*,n})$
for which the infimum (\ref{risk_system_sequential_testing}) is
attained. Formally, let
\begin{align}\notag
&U^n(\pi):=\underset{(\tau^n,d^n)}{\inf}\;Y^n_{(\tau^n,d^n)}(\pi)
\end{align}
be the minimal loss that the DM can achieve and in case that the infimum is
attained, let
\begin{align}\notag
&(\tau^{*,n},d^{*,n}) (\pi) \in
\underset{(\tau^n,d^n)}{\text{arg}\min}\;Y^n_{(\tau^n,d^n)} (\pi)
\end{align}
be an optimal decision rule, given that the prior belief is $\pi$.

We now present the cost function by using $\ptn$. By standard arguments
(see Shiryaev (1978, pp.~166--167)) \cite{Shiryaev1978}, one can
show that the optimal terminal decision $d^{*,n}$ exists and
satisfies $d^{*,n} = l $ if and only if $\ptn({\tau^{*,n}})\geq
\frac{b^n}{a^n+b^n}$. Therefore, we define
\begin{align}
V^n_{\tau^n}(\pi):&=Y^n_{(\tau^n,d^{*,n})}(\pi)\\\notag
&=E^\pi\left[\int_0^{\tau^n} \frac{r}{n}e^{-\frac{r}{n}t}c^n dt  +
\frac{r}{n}e^{-\frac{r}{n}{\tau^n}}(a^n \pn(\tau^n)\wedge b^n
(1-\pn(\tau^n))) \right]
\end{align}
which from Eq.~(\ref{risk_system_general_2}) equals
\begin{align}
V^n_{\tautn}(\pi):=Y^n_{(\tautn,d^{*,n})}(\pi)=E^\pi\left[\int_0^{\tautn}
re^{-rt}c^n dt  + \frac{r}{n}e^{-r{\tautn}}(a^n \ptn(\tautn)\wedge b^n
(1-\ptn(\tautn))) \right].
\end{align}
That is, the cost functions $k^n$ and $K^n$ of the $n$-th system can
be represented as
\begin{align}\label{kn_sequential_testing}
&k^n(\pi)=c^n,\;\;\pi\in[0,1]
\end{align}
and
\begin{align}\label{Kn_sequential_testing}
&K^n(\pi)=a^n \pi\wedge  b^n (1-\pi),\;\;\pi\in[0,1].
\end{align}
Suppose that the limits $\limn a^n/n$, $\limn b^n/n$, and $\limn c^n$
exist and denote them by $a$, $b$, and $c$, respectively. It follows
that $k^n$ and $K^n/n$ converge uniformly on $[0,1]$ to
\begin{align}\label{k_sequential_testing}
&k(\pi)=c,\;\;\pi\in[0,1]
\end{align}
and
\begin{align}\label{K_sequential_testing}
&K(\pi)=a \pi\wedge  b (1-\pi),\;\;\pi\in[0,1],
\end{align}
respectively.
\subsubsection{Asymptotic Optimality}\label{sec:Asymptotic2}
In this section we prove that the optimal stopping time in the
$n$-th system exists uniquely and that it is the first exit
time from an interval. We will therefore restrict the class of the stopping times that we consider to the class of first exit time strategies.
We also show that for every interval $(q_1,q_2)$, the first exit time of the process $(\boldsymbol{\tilde{\pi}}^{\boldsymbol{n}}(t))$
from that interval and the payoff that is associated with this strategy converge to the first exit time of the process $(\boldsymbol{\hat{\pi}}(t))$
from that interval and the payoff that is associated with this strategy, respectively. We conclude this section by finding the asymptotically optimal stopping time and asymptotic value function.

We start with a few properties of the value function $U^n(\pi)$ and
deduce that the optimal strategy in the $n$-th system is a first
exit time strategy. The proof is very similar to the proof of
Theorem 1 in Shiryaev (1978, Ch.~IV) \cite{Shiryaev1978} and
is therefore omitted.
\begin{prop}\label{prop:concave_sequential}
For every fixed $n\inN$, the function $\pi\mapsto U^n(\pi)$ is
bounded from above by $K^n(\pi)/n$, concave, and continuous. Moreover,
$U^n(0)=U^n(1)=0$.
\end{prop}

\begin{rem}\label{rem:Unstrategy}
From Proposition \ref{prop:concave_sequential} it follows that there
are two points $0\leq q_1^{*,n} <  q_2^{*,n}\leq 1$, such that
$U^n(\pi)=(a^n \pi\wedge  b^n (1-\pi))/n=K^n(\pi)/n$ if $\pi\notin
(q_1^{*,n}, q_2^{*,n})$, and $U^n(\pi)<K^n(\pi)$ otherwise. That is,
the optimal strategy is the first exit time from the interval
$(q_1^{*,n}, q_2^{*,n})$ (see the discussion in Shiryaev (1978,
Ch.~IV, pp.~168--169)) \cite{Shiryaev1978}.
\end{rem}
Proposition \ref{prop:concave_sequential} and Remark \ref{rem:Unstrategy} can be formulated for the limit problem as well.
Therefore, one can deduce that there
exists an optimal stopping time that is associated with the continuation region $D^*=(q_1^*, q_2^*)$.
For every $n\inN$, every $\pi\in[0,1]$, and every $q_1<
q_2\in[0,1]$, define the continuation region $(q_1,q_2)$.
The next theorem follows from
Eqs.~(\ref{kn_sequential_testing})--(\ref{K_sequential_testing}),
Theorem \ref{thm:convergence_Vn}, and Corollary \ref{cor:convergenceU}.
\begin{thm}\label{thm:convergence_Vn2}
Fix $0\leq q_1 < q_2\leq 1$. Under Assumptions \ref{asu:t_u_t_v},
\ref{asu:expectation1}, \ref{asu:lambda_rates}, \ref{asu:density1},
and \ref{asu:density22}, the following limits hold:\footnote{As in Section \ref{sec:bandit}, the function
$V_{\taut_{(q_1,q_2)}}(\pi)$ can be expressed explicitly through the
parameters of the problem, but since it has no fundamental
contribution, this expression is omitted.}
\begin{align}\label{eq:convergence_tautnD_to tautD2}
\limn \tautn_{(q_1,q_2)}(\pi) \eqd \taut_{(q_1,q_2)}(\pi),
\end{align}
\begin{align}\label{eq:convergence_Vtautn_to Vtaut_in ecpectation2}
\limn V^n_{\tautn_{(q_1,q_2)}}(\pi) =
V_{\taut_{(q_1,q_2)}}(\pi),
\end{align}
and there are two points $0\leq q_1^* <  q_2^* \leq 1$ such that
\begin{align}\notag
\limn U^n(\pi) = \limn V^n_{\tautn_{(q_1^*,q_2^*)}}(\pi) = V_{\taut_{(q_1^*,q_2^*)}}(\pi) = U(\pi).
\end{align}
\end{thm}
The analog to Remark \ref{rem:distance_vf} to this model holds.

\section{Conclusion}\label{sec:conclusion}
\subsection{Summary} In this paper we studied a problem of estimating a
parameter $\theta$. We started with a sequence of scaled counting processes $\{(\tilde{L}^n_\theta(t))\}_n$
whose distributions depend on an
unknown parameter $\theta$, the prior distribution of which is
known. Moreover, we assumed that $\{(\tilde{L}^n_\theta(t))\}_n$ converges in
distribution to a Brownian motion $(\tilde{L}_\theta(t))$ with an unknown
drift $(\theta t)$. We defined by $(\ptn(t))$ the posterior distribution
process of the parameter $\theta$, given the observations
$(\tilde{L}^n_\theta(s))_{s\leq t}$ and by $(\pt(t))$ the posterior
distribution process of the parameter $\theta$, given the
observations $(\tilde{L}_\theta(s))_{s\leq t}$. We showed that, generally,
$\limn\ptn \neq \pt$, unless the counting processes satisfy a memorylessness
property and no information, regarding the posterior processes, is lost by looking at the present
of the counting processes rather than at their past and present.

We also proved that the limit process $\limn\ptn$ equals to a posterior
distribution process of the process $(\hat{M}_\theta(t))$, which is a Brownian motion with the same unknown drift and a different standard deviation coefficient than the one of $(\tilde{L}_\theta(t))$. Apparently, the difference between the standard deviation coefficients of $(\tilde{L}_\theta(t))$
and $(\hat{M}_\theta(t))$ can be arbitrarily large.
Therefore, we concluded that results concerning
optimal stopping problems w.r.t.~$(\tilde{L}_\theta(t))$ cannot be applied to optimal stopping problems w.r.t.~$(\tilde{L}^n_\theta(t))$,
as the difference in the performance can be arbitrarily bad.



\subsection{Future Directions}
\subsubsection{The Disorder Problem, Diffusion Approximations, and
Queues} The Brownian disorder problem was introduced in Shiryaev
(1978) \cite{Shiryaev1978}.\footnote{This model was generalized in
the context of Brownian motion by, e.g., Vellekoop and Clark (2001)
\cite{Vellekoop2001}, Gapeev and Peskir (2006) \cite{Gapeev2006},
Dayanik (2010) \cite{Dayanik2010}, Sezer (2010) \cite{Sezer2010},
and in the context of other processes different from the Brownian
motion, e.g., Peskir and Shiryaev (2002) \cite{Peskir2002}, Gapeev
(2005) \cite{Gapeev2005}, and Bayraktar, Dayanik, and Karatzas
(2006) \cite{Bayraktar2006}.} In this problem, the drift of a
Brownian motion changes at some unknown and unobservable disorder
time. The objective is to detect this change as quickly as possible
after it happens. This problem is also studied by using the Bayesian
posterior process, that now estimates the probability that the drift
has already changed, based on the past information. I managed to
show that the Bayesian posterior distribution process of a
disorder discrete process that is close in distribution to a
disorder Brownian motion, has a similar structure to the posterior
distribution process in our paper. I
would like to apply this result to optimal stopping-time problem in
the context of a G/G/1 queue under heavy traffic where one of the
parameters of the model such as the arrival/service rate changes
randomly.

I believe that `disorder queues' can enrich the classical models,
as it often happens in real life situations that the parameters of the system change over time.

\subsubsection{Parameter Estimation in General Diffusion Processes}
The structure of the limit process $\limn\ptn$ is
surprising and raises further questions about the structure of
Bayesian posterior distribution processes of more general diffusion
processes with uncertainty. I plan to study an approximation for
a model suggested by Zakai (1969) \cite{Zakai1969}. This model is
fundamental in filtering theory and signal processing. Zakai
analyzed a model with a diffusion process $(X(t))$ satisfying the
stochastic differential equation
\begin{align}
X(t) = X(0) + \int_0^t a(X(s))ds +  \int_0^t b(X(s))dW_1(s) ,
\label{3}
\end{align}
where $X(0)$ is a random variable, $(W_1(t))$ is a Brownian motion,
and $a$ and $b$ are real-valued functions such that $b\neq 0$. Let
$(L(t))$ be the observed process which is related to $(X(t))$ by
\begin{align}
L(t) = \int_0^t g(X(s))ds +  \int_0^t \sigma dW_2(s) ,\label{4}
\end{align}
where $(W_2(t))$ is a Brownian motion, $g$ is a real-valued
function, and $\sigma$ is a positive constant. Notice that if $g$ is
the identity function and if $a=b=0$ then $X(t) \equiv X(0)$ and
$(L(t))$ is a Brownian motion with an unknown linear drift
$(X(0)t)$. This is the model that we studied in this paper with $\theta=X(0)$. Zakai presented an equation that is satisfied
by the unnormalized Bayesian posterior distribution process of the
location of $(X(t))$ given the observation $(L(s))_{0\leq s\leq t}$,
commonly known as the \emph{Zakai equation}, see Zakai (1969, equation (11)) \cite{Zakai1969}. I would like to
consider a sequence of processes $\{(X^n(t),L^n(t))\}_{n\in
\mathbb{N}}$ that converges in distribution to $(X(t),L(t))$ and to
analyze the limit of the Bayesian posterior distribution processes
$$\boldsymbol{p^n}(t,l) := P(\theta=l\mid (L^n(s))_{0\leq s\leq
t}),\;\;t\in[0,\infty),\;l\in S.$$
I would like to see whether the limit of $\boldsymbol{p^n}$ exists, under proper scaling of the parameters, the functions and the processes, and if so, what is its structure and when
can it be considered as the Bayesian posterior distribution process
of another process $(X'(t),L'(t))$ that satisfies Eqs.~(\ref{3})--(\ref{4}) with some $a',b',g',$ and $\sigma'$.
This research can shed a light on the behavior of Bayesian
posterior distribution processes in more general and realistic
models, where the process $(X(t))$ evolves randomly over time.

\textbf{Acknowledgement.} This paper is part of the Ph.D. thesis of the author
done under the supervision of Professor Eilon Solan in Tel-Aviv
University. The author would like to thank Professor Solan for his
many comments that improved the paper. The author is also grateful
to Rami Atar for the discussions we had on the subject, and to two anonymous referees for their
suggestions, which significantly improved the presentation of the paper; Theorem \ref{thm:difference}, Corollary \ref{cor:convergenceU}, and Remark \ref{rem:distance_vf} are a result of their comments.
This research was supported in part by Israel Science Foundation [Grant \#538/11] and by the Google
Inter-university Center for Electronic Markets and Auctions.

\section{APPENDIX - Proofs}\label{sec:proofs}
For the proofs of Theorems \ref{thm:convergence_ptt} and
\ref{thm:convergence_Vn} it is convenient to present a precise
probability space on which the sequence of random parameter systems
is defined.
\subsection{Probability Space}\label{sec:appendix_prob_space}
Let $\theta$ be a random variable defined on the probability space
$(\Omega_\theta, \F_\theta, P_\theta)$. Denote the support of
$\theta$ by $S\subseteq \mathbb{R}$ and suppose that $S$ is bounded
and countable. For every $l\in S$, let $\pi_l:=P_\theta(\theta=l)$.
Let $(\Omega_V, \F_V, P_V)$ be a probability space on which a
sequence of i.i.d. random variables $\{\vi\}_{i\geq 1}$ is defined
such that for every $i\geq 2 $, $\vi$ is distributed as the random
variable $v$ that satisfies Assumption \ref{asu:expectation1}.
Define the probability space $(\Omega_\theta\times \Omega_V,
\sigma(\F_\theta\times\F_V), P)$ such that for every $A_1\times
A_2\subseteq S\times \Omega_V$ one has $P(A_1\times A_2) =
P_\theta(A_1) P_V(A_2)$. Therefore, $\{\vi\}_{i\geq 1}$ and $\theta$
are independent with respect to the probability function $P$.
Denote by $P_l$ the probability measure over $\Omega_\theta\times
\Omega_V$ given $\theta=l$. That is, for every $l\in S$, and every
$C\subseteq S\times \Omega_V$, set $P_l(C):=P (C\mid \theta=l)$.
\subsection{The $n$-th System}
Let $t_u$ be a parameter and $u_1$ be the random variable defined in the previous paragraph, such that for every $t\in[0,\infty)$ one
has
\begin{align}\label{eq:like_asu_t_u}
P(u_1+t_u\geq t)=P(v\geq t \mid v\geq t_u).
\end{align}
For every $l\in S$ and every $n\inN$, define the parameter $\muln$.
For every $i\geq 1$, every $l\in S$, and every $n\inN$ define
$v_{i,l}^n:=\frac{\vi}{\muln}$ and $t^n_v:=\frac{t_u}{\muln}$. For
every $n\in\mathbb{N}$, let
$$\mathcal{RS}^n_{\pi}(\theta)=\left(t^n_v, v, \mun,
\{\vin\}_{i\geq 1}, \pi \right) = \sum_{l\in
S}\mathbb{I}_{\{\theta=l\}}\left(t^n_v, v, \mu_l^n,
\{v^n_{i,l}\}_{i\geq 1} \right)$$ be a sequence of random parameter
systems. For every $n\inN$ this construction generates the random
parameter system that was defined in Section \ref{sec:RPSn}. Notice
that for every $l\in S$ and every $t\in[0,\infty)$ one has
\begin{align}\notag
P_l\left(v_{1,l}^n+t^n_v \geq t\right)
&= P_l\left(\frac{u_1}{\muln}+\frac{t_v}{\muln} \geq t\right)
=P_l\left(u_1+t_v \geq t\muln\right)
= P\left(v\geq t\muln\right|\left.v \geq t_u\right)\\\notag
&=P\left(\frac{v}{\muln}\geq t\right|\left.\frac{v}{\muln} \geq
\frac{t_u}{\muln}\right)
=P\left(\frac{v}{\muln}\geq t\right|\left.\frac{v}{\muln} \geq
t^n_v\right),
\end{align}
where the third equality follows from Eq.~(\ref{eq:like_asu_t_u}).
Therefore, for every $n\inN$ Assumption \ref{asu:t_u_t_v} is
satisfied.

%
%
%
%
%
%

%
\subsection{Proof of Theorem \ref{thm:convergence_ptt}}
We divide the proof into two parts. We first prove
Eq.~(\ref{Main_varphi}) and thereafter conclude Eq.~(\ref{Main_p}).

\subsubsection{\textbf{Proof of Eq.~(\ref{Main_varphi})}}
Recall that by the definition of $\vphis$, the process $\vphiasf$
satisfies
\begin{align}\label{vphit}
\vphiasf(l,t) = \exp \left\{ \frac{l\sigmaf}{\sqrt{\alpha}} W'(t)-
\frac{1}{2}\left(\frac{l\sigmaf}{\sqrt{\alpha}}\right)^2
t+\frac{\theta\sigmaf}{\sqrt{\alpha}}\cdot
\frac{l\sigmaf}{\sqrt{\alpha}} t\right\},\;\;l\in
S,\;t\in[0,\infty),
\end{align}
where $(W'(t))$ is a standard Brownian motion independent of
$\theta$.

In order to prove Eq.~(\ref{Main_varphi}) it suffices to prove that
\begin{align}\label{lim_vphitn}
\limn\ln(\vphitn)\eqd\ln(\vphiasf).
\end{align}
Denote
\begin{align}\notag
\tilde{W}^n(t):=\frac{\sum_{i=1}^{\lfloor nt\rfloor}\left[
-\frac{f'(\vi)}{f(\vi)}\vi\right] -
nt}{\sigmaf\sqrt{n}},\;\;t\in[0,\infty),
\end{align}
\begin{align}\notag
\bar{L}^n(t):=\frac{L^n(nt)}{n},\;\;t\in[0,\infty),
\end{align}
and
\begin{align}\label{Unt}
\zetatn(l,t):=&
\sigmaf\sqrt{n}\frac{\muln-\muzn}{\muthn}\tilde{W}^n(\bar{L}^n(t))\\\notag
&-\frac{1}{2}\left(\frac{l\sigmaf}{\alpha}\right)^2
\bar{L}^n(t)+\frac{\theta\sigmaf}{\alpha}\cdot
\frac{l\sigmaf}{\alpha} \bar{L}^n(t),\;\;l\in S,\;t\in[0,\infty).
\end{align}
From Eq.~(\ref{lim_vphitn}) and Theorem 3.1 in Billingsley (1999)
\cite{Billingsley1999} it follows that in order to prove
Eq.~(\ref{Main_varphi}) it suffices to prove that
\begin{align}\label{lim_vphitn2}
\limn ( \ln(\vphitn) - \zetatn ) = 0 \;\; \text{u.o.c.\;(Proposition
\ref{prop1})}
\end{align}
and
\begin{align}\label{lim_vphitn3}
&\limn \zetatn \eqd \ln(\vphiasf)\;\;\text{(Proposition
\ref{prop2})}.
\end{align}
\begin{prop}[Proving Eq.~(\ref{lim_vphitn2})]\label{prop1}
Under Assumptions \ref{asu:t_u_t_v}, \ref{asu:expectation1},
\ref{asu:lambda_rates}, \ref{asu:density1}, and \ref{asu:density22},
the following holds:
\begin{align}\label{eq:prop1}
\limn ( \ln(\vphitn) - \zetatn ) = 0 \;\; \text{{\rm u.o.c.}}
\end{align}
\end{prop}
\begin{proof}
The following series of equations presents the Radon--Nikod\'{y}m
derivative $(\vphitn(l,t))$ in a more convenient form. For every
$l\in S$ and every $t\in[0,\infty)$ one has
\begin{align}\label{eq:ptn_devided_by_1-ptn}
\vphitn(l,t)
=&
\frac{f_l^n(v_1^n|v_1^n > t^n_v)}{f_0^n(v_1^n|v_1^n>t^n_v)}\cdot
\frac{  \prod_{i=2}^{L^n(nt)} f_l^n(\vin) }{ \prod_{i=2}^{L^n(nt)}
f^n_0(\vin)}
\\\notag
&\cdot  \frac{P_l\left(v_{L^n(nt)+1}^n > nt -
\sum_{i=1}^{L^n(nt)}\vin\mid
\sum_{i=1}^{L^n(nt)}\vin\right)}{P_0\left(v_{L^n(nt)+1}^n > nt -
\sum_{i=1}^{L^n(nt)}\vin\mid
\sum_{i=1}^{L^n(nt)}\vin\right)}\\\notag
=&
\frac{P_0^n(v_1^n > t^n_v)}{P_l^n(v_1^n>t^n_v)}\cdot
\frac{  \prod_{i=1}^{L^n(nt)} f_l^n(\vin) }{ \prod_{i=1}^{L^n(nt)}
f^n_0(\vin)}
\\\notag
&\cdot  \frac{P_l\left(v_{L^n(nt)+1}^n > nt -
\sum_{i=1}^{L^n(nt)}\vin\mid
\sum_{i=1}^{L^n(nt)}\vin\right)}{P_0\left(v_{L^n(nt)+1}^n > nt -
\sum_{i=1}^{L^n(nt)}\vin\mid
\sum_{i=1}^{L^n(nt)}\vin\right)}\\\label{eq:ptn_devided_by_1-ptn2}
=& \exp\left\{\sum_{i=1}^{L^n(nt)}\ln\left(\frac{
f^n_l(\vin)}{f^n_0(\vin)}\right) + \ln\left(\frac{1-F^n_0
(t^n_v)}{1-F^n_l (t^n_v)}\right)\right.\\\notag
&\left.+\ln\left(\frac{1-F^n_l\left( nt -
\sum_{i=1}^{L^n(nt)}\vin\mid
\sum_{i=1}^{L^n(nt)}\vin\right)}{1-F^n_0\left( nt -
\sum_{i=1}^{L^n(nt)}\vin\mid
\sum_{i=1}^{L^n(nt)}\vin\right)}\right)\right\}.
\end{align}
From Eq.~(\ref{eq:ptn_devided_by_1-ptn2}) and the triangle
inequality it follows that, for every $l\in S$ and $t\in[0,\infty)$,
\begin{align}\label{eq:lnvarphi_U}
&|\ln(\vphitn)(l,t) - \zetatn(l,t)|\\\notag
&\leq\left| \sum_{i=1}^{L^n(nt)}\ln\left(\frac{
f^n_\theta(\vin)}{f^n_0(\vin)}\right) + \ln\left(\frac{1-F^n_0
(t^n_v)}{1-F^n_\theta (t^n_v)}\right)
+\ln\left(\frac{1-F^n_\theta\left( nt - \sum_{i=1}^{L^n(nt)}\vin\mid
\sum_{i=1}^{L^n(nt)}\vin\right)}{1-F^n_0\left( nt -
\sum_{i=1}^{L^n(nt)}\vin\mid \sum_{i=1}^{L^n(nt)}\vin\right)}\right)
\right.\\\notag
&-\left.
\sigmaf\sqrt{n}\frac{\muzn-\muthn}{\muthn}\tilde{W}^n(\bar{L}^n(t))
-\frac{1}{2}\left(\frac{(0-\theta)\sigmaf}{\alpha}\right)^2
\bar{L}^n(t) \right|\\\notag
&+\left| \sum_{i=1}^{L^n(nt)}\ln\left(\frac{
f^n_\theta(\vin)}{f^n_l(\vin)}\right) - \ln\left(\frac{1-F^n_l
(t^n_v)}{1-F^n_\theta (t^n_v)}\right)
-\ln\left(\frac{1-F^n_\theta\left( nt - \sum_{i=1}^{L^n(nt)}\vin\mid
\sum_{i=1}^{L^n(nt)}\vin\right)}{1-F^n_l\left( nt -
\sum_{i=1}^{L^n(nt)}\vin\mid \sum_{i=1}^{L^n(nt)}\vin\right)}\right)
\right.\\\notag
&-\left.
\sigmaf\sqrt{n}\frac{\muln-\muthn}{\muthn}\tilde{W}^n(\bar{L}^n(t))
-\frac{1}{2}\left(\frac{(l-\theta)\sigmaf}{\alpha}\right)^2
\bar{L}^n(t) \right|.
\end{align}
We prove that the second term on the right-hand side of
Eq.~(\ref{eq:lnvarphi_U}) converges to zero u.o.c. The proof for the
first term is similar and is therefore omitted. From the triangle
inequality it follows it is sufficient to verify that the following
two processes converge to zero u.o.c.:
\begin{align}\label{eq:lnvarphi_U3a}
\xitn(l,t):= \ln\left(\frac{1-F^n_l (t^n_v)}{1-F^n_\theta
(t^n_v)}\right)+ \ln\left(\frac{1-F^n_\theta\left( nt -
\sum_{i=1}^{L^n(nt)}\vin\mid
\sum_{i=1}^{L^n(nt)}\vin\right)}{1-F^n_l\left( nt -
\sum_{i=1}^{L^n(nt)}\vin\mid
\sum_{i=1}^{L^n(nt)}\vin\right)}\right),
\end{align}
$l\in S,\;t\in[0,\infty),$ and
\begin{align}\label{eq:lnvarphi_U3b}
\chitn(l,t):=\sum_{i=1}^{n\bar{L}^n(t)}\ln\left(\frac{
f^n_\theta(\vin)}{f^n_l(\vin)}\right)
-\sigmaf\sqrt{n}\frac{\muln-\muthn}{\muthn}\tilde{W}^n(\bar{L}^n(t))
-\frac{1}{2}\left(\frac{(l-\theta)\sigmaf}{\alpha}\right)^2
\bar{L}^n(t),
\end{align}
$l\in S,\;t\in[0,\infty)$. We prove these convergence in Lemma
\ref{lemma0} and Lemma \ref{lemma1}, respectively.
\begin{lem}\label{lemma0}
Under Assumptions \ref{asu:t_u_t_v}, \ref{asu:expectation1},
\ref{asu:lambda_rates}, and \ref{asu:density22}.3,
\begin{align}\label{eq:lemma0}
\limn\xitn = 0 \;\; \text{{\rm u.o.c.}}
\end{align}
\end{lem}
\begin{proof}
To prove Eq.~(\ref{eq:lemma0}) it suffices to show that for every
$T>0$ the following two equalities hold:
\begin{align}\label{limAn0_part_1b}
P\left(\limn
\underset{S\times[0,T]}{\sup}\,\left|\ln\left(\frac{1-F^n_l
(t^n_v)}{1-F^n_\theta (t^n_v)}\right) \right|= 0\right)=1
\end{align}
and
\begin{align}\label{limAn0_part_2b}
P\left(\limn \underset{S\times
[0,T]}{\sup}\,\left|\ln\left(\frac{1-F^n_l\left( nt -
\sum_{i=1}^{L^n(nt)}\vin\mid
\sum_{i=1}^{L^n(nt)}\vin\right)}{1-F^n_\theta\left( nt -
\sum_{i=1}^{L^n(nt)}\vin\mid \sum_{i=1}^{L^n(nt)}\vin\right)}\right)
\right|= 0\right)=1.
\end{align}

We prove only Eq.~(\ref{limAn0_part_2b}). The proof of
Eq.~(\ref{limAn0_part_1b}) is similar and is therefore omitted.
The following series of equations, which holds for sufficiently
large $n\inN$, yields an upper bound for the expression
$\underset{S\times [0,T]}{\sup}\,\left|\ln\left(\frac{1-F^n_l\left(
nt - \sum_{i=1}^{L^n(nt)}\vin\mid
\sum_{i=1}^{L^n(nt)}\vin\right)}{1-F^n_\theta\left( nt -
\sum_{i=1}^{L^n(nt)}\vin\mid \sum_{i=1}^{L^n(nt)}\vin\right)}\right)
\right|$:

\begin{align}\notag
& \underset{S\times [0,T]}{\sup}\,\left|\ln\left(\frac{1-F^n_l\left(
nt - \sum_{i=1}^{L^n(nt)}\vin\mid
\sum_{i=1}^{L^n(nt)}\vin\right)}{1-F^n_\theta\left( nt -
\sum_{i=1}^{L^n(nt)}\vin\mid \sum_{i=1}^{L^n(nt)}\vin\right)}\right)
\right| \\\label{eeeq1}
=&\underset{S\times [0,T]}{\sup}\,
\left|
\ln\left(1-F\left( \muln\left.\left(nt -
\sum_{i=1}^{L^n(nt)}\vin\right)\right|\;
\sum_{i=1}^{L^n(nt)}\vin\right)\right)\right.\\\notag
&\left.-\ln \left(1-F\left.\left( \muthn\left(nt-
\sum_{i=1}^{L^n(nt)}\vin\right)\right|\;
\sum_{i=1}^{L^n(nt)}\vin\right)\right) \right| \\\label{eeeq2}
=&\underset{S\times [0,T]}{\sup}\,|\muthn-\muln|\left(nt -
\sum_{i=1}^{L^n(nt)}\vin\right)\frac{f\left(d^n_l\left(nt -
\sum_{i=1}^{L^n(nt)}\vin\right)\right)}{1-F\left(d^n_l\left(nt -
\sum_{i=1}^{L^n(nt)}\vin\right)\right)}\\\label{eeeq3}
\leq& \,\underset{S\times
[0,T]}{\sup}\,\frac{\sqrt{n}|\muthn-\muln|}{d^n_l}\frac{1}{\sqrt{n}}N
\left(d^n_l \left(nt -
\sum_{i=1}^{L^n(nt)}\vin\right)\right)\\\label{eeeq4}
\leq&
\,\underset{S}{\sup}\,\frac{\sqrt{n}|\muthn-\muln|}{d^n_l}\cdot\underset{
[0,T]}{\sup}\,\frac{1}{\sqrt{n}}N \left((1+\epsilon_N)\muthn
v_{L^n(nt)+1}^n \right)\\\label{eeeq5}
=&
\,\underset{S}{\sup}\,\frac{\sqrt{n}|\muthn-\muln|}{d^n_l}\cdot\underset{
[0,T]}{\sup}\,\frac{1}{\sqrt{n}}N \left(
(1+\epsilon_N)u_{L^n(nt)+1}\right),
\end{align}
where $d^n_l\in(\muthn,\muln)$ or $d^n_l\in(\muln,\muthn)$.
Eq.~(\ref{eeeq1}) follows from Remark
\ref{rem:density_f_to_f_theta}, while Eq.~(\ref{eeeq2}) follows from
the Lagrange mean value theorem. Inequality (\ref{eeeq3}) follows
from Assumption \ref{asu:density22}.3 and the fact that $N(x)$ is
monotone nondecreasing. Inequality (\ref{eeeq4}) follows since, by
Eq.~(\ref{eq:A2}),
\begin{align}
\limn\underset{S}{\sup}\,|d^n_l-\muthn|\leq\limn\underset{S}{\sup}\,|\muln-\muthn|=0,
\end{align} and since $N(x)$ is monotone nondecreasing.
Eq.~(\ref{eeeq5}) follows since for every $i\geq1$ and every $n\inN$
one has $\muthn\vin=\vi$ (see Section
\ref{sec:appendix_prob_space}). Assumption \ref{asu:density22}.3
implies that $E[N((1+\epsilon_N)v)]^2<\infty$ and therefore
\begin{align}\label{eeeq51}
\limn\frac{1}{\sqrt{n}}N \left((1+\epsilon_N)u_{L^n(nt)+1} \right) =
\limn\sqrt{\tfrac{1}{n}N^2 \left((1+\epsilon_N)u_{L^n(nt)+1}
\right)}=0\;\;\text{u.o.c.}
\end{align}
Finally, from Assumption \ref{asu:lambda_rates}.1 it follows that
\begin{align}\label{eeeq52}
\limn \underset{S\times[0,T]}{\sup}\, \sqrt{n}|\muln-\muthn|
\leq& \limn \underset{S}{\sup}\, |\sqrt{n}(\muln-\muzn) - l|
+\underset{S}{\sup}\, |\sqrt{n}(\muthn-\muzn) - \theta|\\\notag
&+ \underset{S}{\sup}\, |l|+
 |\theta| <\infty,
\end{align}
where the last inequality follows since $S$ is bounded.
Eqs.~(\ref{eeeq51})--(\ref{eeeq52}) imply that the right-hand side
of Eq.~(\ref{eeeq5}) converges to $0$ u.o.c.

\end{proof}
\begin{lem}\label{lemma1}
Under Assumptions \ref{asu:expectation1}, \ref{asu:lambda_rates},
\ref{asu:density1}, \ref{asu:density22}.1, and
\ref{asu:density22}.2,
\begin{align}\label{eq:lemma1}
\limn\chitn(l,t) = 0 \;\; \text{{\rm u.o.c.}}
\end{align}
\end{lem}
\begin{proof}
Eq.~(\ref{eq:lemma1}) is equivalent to the requirement that
\begin{align}\label{limXnt_1}
P\left(\limn \underset{S\times[0,T]}{\sup}\, \left|\chitn(l,t)
\right|= 0\right)=1
\end{align}
for every $T>0$. The first term in Eq.~(\ref{eq:lnvarphi_U3b}) is
$\sum_{i=1}^{\bar{L}^n(t)}\ln\left(\frac{
f^n_\theta(\vin)}{f^n_l(\vin)}\right) $ which is a composition of
$$\sum_{i=1}^{\lfloor nt\rfloor}\ln\left(\frac{
f^n_\theta(\vin)}{f^n_l(\vin)}\right) $$ and $\bar{L}^n(t)$. For
every $l\in S$ denote
\begin{align}
\hn:=\frac{\muln-\muthn}{\muthn}.
\end{align}
The following series of equations presents $\sum_{i=1}^{\lfloor
nt\rfloor}\ln\left(\frac{ f^n_\theta(\vin)}{f^n_l(\vin)}\right) $ in
a more convenient form:
\begin{align}\label{eq1}
\sum_{i=1}^{\lfloor nt\rfloor}&\ln\left(\frac{
f^n_\theta(\vin)}{f^n_l(\vin)}\right)
=\sum_{i=1}^{\lfloor nt\rfloor}\ln\left(\frac{ \muthn
f(\muthn\vin)}{\muln f(\muln\vin)}\right)\\\label{eq2}
=& - nt\ln (1 + \hn) - \sum_{i=1}^{\lfloor nt\rfloor} \left[
\ln\left( f(\muln\vin) \right) - \ln\left( f(\muthn\vin)
\right)\right]\\\label{eq3}
=& - nt \ln (1+\hn) - \sum_{i=1}^{\lfloor nt\rfloor}\left[ \ln\left(
f(\vi + \vi \hn) \right) - \ln\left( f(\vi)
\right)\right]\\\label{eq4}
=& - nt\left(\hn-\frac{1}{2}(\hn)^2 \right) -  nt\left(\ln (1+\hn)
-\hn+\frac{1}{2}(\hn)^2 \right)
\\\notag
&-\sum_{i=1}^{\lfloor nt\rfloor}\left[ \frac{f'(\vi)}{f(\vi)}\vi\hn
+ \frac{1}{2!}\left(\frac{f'(\vi)}{f(\vi)}\right)'\vi^2(\hn)^2
+ \frac{1}{3!}\left(\frac{f'(\cin)}{f(\cin)}\right)''\vi^3(\hn)^3
\right]\\\label{eq6}
=&\,\sigmaf\sqrt{n}\hn\tilde{W}^n(t)  +  \frac{1}{2}(\sqrt{n}\hn)^2
\left(  t  - \tfrac{1}{n} \sum_{i=1}^{\lfloor
nt\rfloor}\left(\frac{f'(\vi)}{f(\vi)}\right)'\vi^2 \right) \\\notag
&+ \frac{(\sqrt{n}\hn)^3}{3!}\cdot \sum_{i=1}^{\lfloor
nt\rfloor}\frac{1}{n^{1.5}}\left(\frac{f'(\cin)}{f(\cin)}\right)''\vi^3
+ (\sqrt{n}\hn)^2\left(\frac{\ln (1+\hn) -\hn+\frac{1}{2}(\hn)^2
}{(\hn)^2}\right)t,
\end{align}
where $\cin\in(\vi,\vi+\vi\hn)$ or $\cin\in(\vi+\vi\hn,\vi)$.
Eq.~(\ref{eq2}) follows from Remark \ref{rem:density_f_to_f_theta}
and the definition of $\hn$. Eq.~(\ref{eq3}) follows by the
definition of $\vi$. Since $f\in \mathcal{C}^3$, Eq.~(\ref{eq4}) follows from
the Taylor expansion of the function $\ln(f(x))$ with Lagrange
remainder of order $3$.
Eq.~(\ref{eq6}) is merely a rearrangement of the terms. From
Eqs.~(\ref{eq:lnvarphi_U3b}) and (\ref{eq6}) it follows that for
every $l\in S$ and $t\in[0,\infty)$ one has
\begin{align}\label{Xnt_2}
\chitn(l,t) =&\sum_{i=1}^{n\bar{L}^n(t)} \ln\left(\frac{
f^n_\theta(\vin)}{f^n_l(\vin)}\right)
-\sigmaf\sqrt{n}\hn\tilde{W}^n(\bar{L}^n(t))
-\frac{1}{2}\left(\frac{(l-\theta)\sigmaf}{\alpha}\right)^2
\bar{L}^n(t)\\\label{Xnt_3}
=&\left[-\frac{1}{2}\left(\frac{(l-\theta)\sigmaf}{\alpha}\right)^2
\bar{L}^n(t) + \frac{1}{2}(\sqrt{n}\hn)^2 \left(  \bar{L}^n(t)  -
\tfrac{1}{n}
\sum_{i=1}^{n\bar{L}^n(t)}\left(\frac{f'(\vi)}{f(\vi)}\right)'\vi^2
\right)\right] \\\notag
&+\frac{(\sqrt{n}\hn)^3}{3!}\cdot\sum_{i=1}^{\lfloor
nt\rfloor}\frac{1}{n^{1.5}}\left(\frac{f'(\cin)}{f(\cin)}\right)''\vi^3\\\notag
&+(\sqrt{n}\hn)^2\left(\frac{\ln (1+\hn) -\hn+\frac{1}{2}(\hn)^2
}{(\hn)^2}\right)\bar{L}^n(t)  .
\end{align}
We are now ready to prove Eq.~(\ref{eq:lemma1}). We show that each
of the three terms on the right-hand side of Eq.~(\ref{Xnt_3})
converges to zero u.o.c.
$\\ $\textbf{Part I: First term.} Define the following functions and
processes:
\begin{align}
g_1(l):=\frac{l-\theta}{\alpha},\;\;l\in S,
\end{align}
\begin{align}
g_1^n(l):=\sqrt{n}\hn = \sqrt{n}\frac{\muln-\muthn}{\muthn},\;\;l\in
S,
\end{align}
\begin{align}
G_1(l):=(1-\sigmaf^2) t,\;\;t\in [0,\infty),
\end{align}
\begin{align}
G_1^n(l):=\tfrac{1}{n} \sum_{i=1}^{\lfloor
nt\rfloor}\left(\frac{f'(\vi)}{f(\vi)}\right)'\vi^2,\;\;t\in
[0,\infty),
\end{align}
and
\begin{align}
\bar{L}(t)=\alpha t,\;\;t\in [0,\infty).
\end{align}
Therefore, the first term in Eq.~(\ref{Xnt_2}) can be expressed as
\begin{align}
-\frac{1}{2}\left(g_1(l)\right)^2\sigmaf^2 \bar{L}^n(t) +
\frac{1}{2}(g^n_1(l))^2 \left(  \bar{L}^n(t)  -
G_1^n(\bar{L}^n(t))\right),\;\;l\in S,\;t\in [0,\infty).
\end{align}

From the definition of $\hn$ it follows that
\begin{align}\label{eq:a}
(\sqrt{n}\hn) &= \left(\frac{\sqrt{n}(\muln-\muthn)}{\muthn}\right)
=
\left(\frac{\sqrt{n}(\muln-\muzn)}{\muthn}+\frac{\sqrt{n}(\muthn-\muzn)}{\muthn}\right).
\end{align}
Assumption \ref{asu:lambda_rates} and Eq.~(\ref{eq:a}) implies that
\begin{align}\label{eq:b}
\limn g_1^n=g_1\;\;\text{u.o.c.}
\end{align}
From Lemma \ref{lem:R} (Eq.~(\ref{eq:R2})) and the Functional Strong
Law of Large Numbers (FSLLN, see Chen and Yao (2001, Theorem 5.10)
\cite{Chen2001}) it follows that
\begin{align}\label{eq:c}
\limn G_1^n = G_1 \;\;\text{u.o.c.}
\end{align}
Next, Assumption \ref{asu:lambda_rates}.2 and the FSLLN imply that
\begin{align}\label{eq:d}
\limn\bar{L}^n=\bar{L}\;\;\text{u.o.c.}
\end{align}
and therefore, by Eqs.~(\ref{eq:c}) and (\ref{eq:d}) and the
random time-change theorem (Chen and Yao (2001, Theorem 5.3)
\cite{Chen2001}),
\begin{align}\label{eq:e}
\limn G_1^n(\bar{L}^n) = G_1(\bar{L}) \;\;\text{u.o.c.}
\end{align}
Therefore, from Eqs.~(\ref{eq:b}),(\ref{eq:d}), and (\ref{eq:e}) it
follows that
\begin{align}
\limn\left[-\frac{1}{2}\left(g_1\right)^2\sigmaf^2 \bar{L}^n +
\frac{1}{2}(g^n_1)^2 \left(  \bar{L}^n  -
G_1^n(\bar{L}^n)\right)\right]=0\;\;\text{u.o.c.}
\end{align}

$\\ $\textbf{Part II: Second term.} Define the process
\begin{align}
G_2^n(t):=\frac{1}{n^{1.5}}\sum_{i=1}^{\lfloor
nt\rfloor}\left[\left(\frac{f'(\cin)}{f(\cin)}\right)''\vi^3(\hn)^3
\right],\;\;t\in [0,\infty).
\end{align}
Therefore, the second term can be expressed as
\begin{align}
\frac{(g_1^n(l))^3}{3!}G^n_2(\bar{L}^n(t)),\;\;l\in S,\;t\in
[0,\infty).
\end{align}
The following equations hold for sufficiently large $n$:
\begin{align}\label{eq:o(1)3}
&G^n_2(t)=\frac{1}{n^{1.5}}\sum_{i=1}^{\lfloor
nt\rfloor}\left|\left(\frac{f'(\cin)}{f(\cin)}\right)''\vi^3 \right|
\leq \frac{1}{(1-\epsilon_M)^3}\frac{1}{n^{1.5}}\sum_{i=1}^{\lfloor
nt\rfloor}M((1+\epsilon_M)\vi)
\end{align}
%
The inequality in Eq.~(\ref{eq:o(1)3}) follows from Assumption
\ref{asu:density22}.2 since $\cin\in(\mukn\vin,\muln\vin)$ or
$\cin\in(\muln\vin,\mukn\vin)$. Eq.~(\ref{eq:A2}) implies that for
sufficiently large $n\inN$ and every $l\in S$ one has
$(1-\epsilon_M)\vi\leq \cin\leq (1+\epsilon_M)\vi$. From the FSLLN
and Eq.~(\ref{eq:o(1)3}) it follows that
\begin{align}\label{eq:o(1)4}
\lim G^n_2(t) = 0\;\;\text{u.o.c.}
\end{align}
Now Eqs.~(\ref{eq:b}), (\ref{eq:d}), and (\ref{eq:o(1)4}) and the
random time-change theorem (Chen and Yao (2001, Theorem 5.3)
\cite{Chen2001}) yield that
\begin{align}
\limn\frac{g_1^n}{3!}G^n_2(\bar{L}^n) = 0,\;\;\text{u.o.c.}
\end{align}

$\\ $\textbf{Part III: Third term.} Define the function
\begin{align}
g_2^n(l)=\frac{\ln (1+\hn) -\hn+\frac{1}{2}(\hn)^2
}{(\hn)^2},\;\;l\in S.
\end{align}
Therefore, the third term can be expressed as
\begin{align}
(g_1^n(l))^2 g_2^n(l)\bar{L}^n(t) ,\;\;l\in S,\;t\in[0,\infty).
\end{align}

From the Taylor expansion of $\ln(1+x)$ and Eq.~(\ref{eq:b}) and
(\ref{eq:d}) it follows that
\begin{align}\notag
& \limn(g_1^n)^2 g_2^n\bar{L}^n  = 0 \;\;\text{u.o.c.}
\end{align}
This completes the proof of Lemma \ref{lemma1}.
\end{proof}

This completes the proof of Proposition \ref{prop1}.
\end{proof}

\begin{prop}[Proving Eq.~(\ref{lim_vphitn3})]\label{prop2}
Under Assumptions \ref{asu:lambda_rates} and \ref{asu:density22}.1,
\begin{align}\notag
&\limn \zetatn \eqd \ln(\vphiasf).
\end{align}

\end{prop}
\begin{proof}
From Eq.~(\ref{eq:h}) it follows that for every $n\inN$ the process
$\zetatn$ can be expressed as
\begin{align}\notag
\zetatn(l,t)=& \sigmaf\sqrt{n}h^n(l)\tilde{W}^n(\bar{L}^n(t))
-\frac{1}{2}\left(\frac{I_S(l)\sigmaf}{\alpha}\right)^2
\bar{L}^n(t)\\\notag &+\frac{\theta\sigmaf}{\alpha}\cdot
\frac{I_S(l)\sigmaf}{\alpha} \bar{L}^n(t),\;\;l\in
S,\;t\in[0,\infty).
\end{align}

We prove that there exists a probability space $\Omega_W$ such that
\begin{align}\label{U_u.o.c.}
&\limn \zetatn= \ln(\vphiasf)\;\;\text{u.o.c.}
\end{align}
in the probability space $\Omega_\theta\times\Omega_W$. For this, we
investigate separately the parts of the process $\zetatn$ that
depend on $\theta$ and the parts that depend on $\{\vi\}_{i\geq 1}$.
From Assumption \ref{asu:lambda_rates} it follows that
\begin{align}\label{U_u.o.c.1}
\limn\sigmaf\sqrt{n}h^n = \sigmaf\frac{I_S}{\alpha}\;\;\text{u.o.c.}
\end{align}
The processes $(\bar{L}^n(t))$ and $(\tilde{W}^n(\bar{L}^n(t)))$
depend on $\{\vi\}_{i\geq 1}$, which is independent of $\theta$.
From the Skorokhod Representation Theorem and the random time-change theorem (see Chen and Yao (2001, Theorems 5.1 and 5.3)
\cite{Chen2001}) it follows that there exist a probability space
$\Omega_W$ and a standard Brownian motion $(\tilde{W}(t))$ defined
on $\Omega_W$, such that
\begin{align}\label{U_u.o.c.2}
\limn(\bar{L}^n,\tilde{W}^n(\bar{L}^n))=(\bar{L},\tilde{W}(
\bar{L}))\;\;\text{u.o.c.}
\end{align}
From Eqs.~(\ref{U_u.o.c.1}) and (\ref{U_u.o.c.2}) it follows that in
the probability space $\Omega_\theta\times\Omega_W$
\begin{align}\label{Unt1}
\limn\zetatn&= \limn\sigmaf\sqrt{n}h^n\tilde{W}^n(\bar{L}^n)
-\frac{1}{2}\left(\frac{I_S\sigmaf}{\alpha}\right)^2
\bar{L}^n+\frac{\theta\sigmaf}{\alpha}\cdot
\frac{I_S\sigmaf}{\alpha} \bar{L}^n\\\notag &=
\sigmaf\frac{I_S}{\alpha}\tilde{W}(\bar{L})
-\frac{1}{2}\left(\frac{I_S\sigmaf}{\alpha}\right)^2 \bar{L} +
\frac{\theta\sigmaf}{\alpha}\cdot \frac{I_S\sigmaf}{\alpha}
\bar{L}\;\;\text{u.o.c.}
\end{align}
and since convergence u.o.c. implies convergence in distribution,
\begin{align}\label{Unt2}
\limn\zetatn\eqd \sigmaf\frac{I_S}{\alpha}\tilde{W}( \bar{L})
-\frac{1}{2}\left(\frac{I_S\sigmaf}{\alpha}\right)^2 \bar{L} +
\frac{\theta\sigmaf}{\alpha}\cdot \frac{I_S\sigmaf}{\alpha} \bar{L}
\end{align}
The scaling of the standard Brownian motion implies that
$(\tilde{W}(\bar{L} (t)))$ is distributed as
$(\sqrt{\alpha}\tilde{W}(t))$ and the result follows.

%

This completes the proof of Eq.~(\ref{Main_varphi}).
\end{proof}

The following remark explains the requirement that the appropriate
rates under the different types are relatively close, up to order
 $\frac{1}{\sqrt{n}}$ (Assumption \ref{asu:lambda_rates}.1).
\begin{rem}\label{rem2}
If there exists a parameter value $l^*\in S$ such that the
difference between the rates $\mu_{l^*}^n$ and $\muzn$ satisfies
$|\mu_{l^*}^n - \muzn|>> \frac{1}{\sqrt{n}},$
then for every $t>0$ the following limit
holds: $\limn\sigmaf\sqrt{n}h^n(t,l^*) = \pm\infty$, and there will
be no convergence of $\zetatn(t,l^*)$. On the other hand, if there
is a parameter value $l^*\in S$ such that the difference between the
rates $\muln$ and $\muzn$ satisfies $|\mu_{l^*}^n - \muzn|<< \frac{1}{\sqrt{n}},$
then for every $t>0$ the following limit holds:
$\limn\sigmaf\sqrt{n}h^n(t,l^*) = 0$, and the DM will not be able to
distinguish between them.
\end{rem}

\subsubsection{\textbf{Proof of Formula (\ref{Main_p}).}}
From Eq.~(\ref{frac_q_l_q_m}) we have
\begin{align}\notag
\piasf (l,t): = \frac{\pi_l\vphiasf(l,t)}{\sum_{k\in S}
 \pi_k\vphiasf(k,t)},\;\;l\in S,\;t\in[0,\infty).
\end{align}
We show that
\begin{align}\notag
\limn\ptn\eqd\piasf.
\end{align}
To this end we define a function
$\Lambda:\mathcal{E}_\infty\rightarrow\mathcal{E}_\infty$ by
\begin{align}\label{Lambda}
\Lambda(\vphi)(l,t) :=  \frac{\pi_l\vphi(l,t)}{\sum_{k\in S}
 \pi_k\vphi(k,t)},\;\;l\in S,\;t\in[0,\infty).
\end{align}
$\Lambda$ is continuous with respect to the metric $e_\infty$.
Therefore,
\begin{align}\notag
\limn\ptn=\limn \Lambda(\vphitn) \eqd \Lambda(\vphiasf)=\piasf ,
\end{align}
where the first equality follows from Eqs.~(\ref{frac_p_l_p_m}) and
(\ref{Lambda}), and the second equality follows from
Eq.~(\ref{Main_varphi}). This completes the proof of Theorem
\ref{thm:convergence_ptt}.

\subsection{Proof of Theorem \ref{thm:difference}}
In order to construct a random variable for which the difference $\sigma_{v} - \frac{1}{\sigma_{f}}$ is large, we
use a random variable that has expectation $1$ and has no variance. Let $z$ be a random variable with the density $g(x):=C/(1+x^3)$, $x>0$, where $C=2\pi/3^{1.5}$.
Then, $E[z]=1$ and ${\rm Var}[z]=\infty$. We now show that $z$ satisfies Assumption \ref{asu:density22}. The variance of $\frac{g'(z)}{g(z)}z$ is given by
\begin{align}\label{eq:int0}
\sigmaf^2 = \int_0^\infty \left(\frac{g'(x)}{g(x)}x\right)^2 g(x) dx
          = \int_0^\infty \left(\frac{3x^3}{1+x^3}\right)^2 g(x) dx <\infty
\end{align}
and there exists a constant $D_1$ such that for every $x>0$
\begin{align}\label{eq:D_1}
\left|\left(\frac{g'(x)}{g(x)}\right)''x^3\right|, \left|\frac{xg(x)}{1-G(x)}\right|\leq D_1,
\end{align}
where $G$ is the cdf of $z$. The random variable $z$ fails to satisfy Assumption \ref{asu:expectation1}.2 since
${\rm Var}[z]=\infty$. Let $\mathbb{I}_A(x)$ be a function that equals $1$ if $x\in A$ and $0$ otherwise and fix $y>1$. We now construct a $y$-dependent random variable that satisfies Assumptions \ref{asu:expectation1}, \ref{asu:density1}, and \ref{asu:density22}, whose density is `similar' to the function
$g_y(x) := g(x)\mathbb{I}_{\{ 0 < x < y \}}(x) + e^{-x}\mathbb{I}_{\{ y < x \}}(x)$, and for which the difference $\sigma_{v} - \frac{1}{\sigma_{f}}$ is large. One may notice that for sufficiently large $y$, the function $g_y$ is not a density function since
$\int_0^\infty g_y(x)dx < 1$. Moreover, for large $y$'s the `expectation' is not one as $\int_0^\infty x g_y(x)dx < 1$. In order to construct a density `similar' to $g_y$ we add to $g_y$ a function $h_y$ that is a sum of two functions. Each of these two
functions has a significant contribution only to one of the two integrals mentioned above.
Let $u,C_2,C_3$ be positive constants and define the function $h(x): =u C_1  \mathbb{I}_{\{1/u < x < 2/u\}}(x) + u C_2  \mathbb{I}_{\{u < x < u+1/u^2\}}(x)$. For a sufficiently large $u$ one has
$\int_0^\infty h(x)dx =  C_1 + C_2/u  \approx C_1$ and $\int_0^\infty xh(x)dx = 3C_1/2u + C_2(1+1/2u^3) \approx C_2$. Therefore, for sufficiently large $y$ one can construct a $\mathcal{C}^3$ function $e_y$ that satisfies the following conditions:
\begin{itemize}
\item[(C1)] $e_y\approx g_y+h_y$, where $h_y$ admits the same form as $h$ with some proper $y$-dependent parameters $u,C_1$, and $C_2$, where $u>4$ for every $y$,
\item[(C2)] $\int_0^\infty e_y(x)dx = 1 $;
\item[(C3)] $\int_0^\infty x e_y(x)dx = 1 $;
\item[(C4)] there exists $w:=w_y >0$ such that for every $x>w$ one has $e_y(x)=g_y(x)=e^{-x}$;
\item[(C5)] for every $2/3<x<1$ one has $e_y(x)=g_y(x)$;\;\; and
\item[(C6)] there exists a positive parameter $D_2$ such that $\left|\frac{e_y'(x)}{e_y(x)}x\right|$, $\left|\left(\frac{e_y'(x)}{e_y(x)}\right)''x^3\right|$, and $\left|\frac{xe_y(x)}{1-E_y(x)}\right|$ are bounded from above by $D_2$ on the interval $(0,w)$, where $E_y$ is the cdf that is associated with the pdf $e_y$.
\end{itemize}
Conditions (C1)--(C3) can hold by the preceding discussion.
To see why one can choose $e_y$ that satisfies Conditions (C4) and (C5) notice that $h_y$ is nonzero only over $(1/u,2/u)\cup(u,u+1/u^2)$.
Therefore, $e_y$ can be chosen to be equal to $g_y$ on any subinterval of the complement of $(1/u,2/u)\cup(u,u+1/u^2)$. Condition (C4) can hold by taking $w=u+1/u^2$, and Condition (C5) can hold since $u>4$ for every $y$ by Condition (C1). Condition (C6) can hold by Eq.~(\ref{eq:D_1}) and by Condition (C1).

Let $v:=v_y$ be a random variable that is associated with the pdf $e_y$. We show that $v$ satisfies Assumptions \ref{asu:expectation1} and \ref{asu:density22}.
The variance of $v$ is finite since
\begin{align}\label{eq:vy}
\sigma_{v[y]}^2 &= \int_0^\infty x^2 e_y(x)dx
\approx \int_0^\infty x^2(g_y(x) + h_y(x))dx\\\notag
&= \int_0^y x^2g(x)dx +\int_y^\infty x^2e^{-x}dx + \int_0^\infty x^2 h_y(x)dx <\infty.
\end{align}
The variance of $\frac{e_y'(v)}{e_y(v)}v$ is also finite since by Conditions (C5) and (C6) one has
\begin{align}\notag
\sigma_{f[y]}^2 &= \int_0^\infty \left(\frac{e_y'(x)}{e_y(x)}x\right)^2 e_y(x) dx
=  \int_0^w \left(\frac{e_y'(x)}{e_y(x)}x\right)^2 e_y(x) dx +  \int_w^\infty \left(\frac{e_y'(x)}{e_y(x)}x\right)^2 e_y(x) dx\\\notag
&\leq  D_2 w +e^{-w} <\infty.
\end{align}
Assumptions \ref{asu:density22}.2 and \ref{asu:density22}.3 also follow by Conditions (C5) and (C6).

We now show that by taking large $y$'s, the difference $\sigma_{v[y]}-\frac{1}{\sigma_{f[v]}}$ becomes large.
To this end, we show that $\limy \left(\sigma_{v[y]}-\frac{1}{\sigma_{f[v]}}\right)=\infty$. Let $y$ be such that Conditions (C1)--(C6) hold. As in Eq.~(\ref{eq:vy}) one has
\begin{align}\label{eq:vy2}
\sigma_{v[y]}^2\approx \int_0^\infty x^2(g_y(x) + h_y(x))dx \geq \int_0^y x^2 g_y(x)dx = \int_0^y x^2g(x)dx = C\ln(1+y^3)/3.
\end{align}
By Condition (C5) one has
\begin{align}\notag
\sigma_{f[y]}^2 &= \int_0^\infty \left(\frac{e_y'(x)}{e_y(x)}x\right)^2 e_y(x) dx
\geq  \int_{2/3}^1 \left(\frac{e_y'(x)}{e_y(x)}x\right)^2 e_y(x) dx =\int_{2/3}^1 \left(\frac{g'(x)}{g(x)}x\right)^2 g(x) dx   \\\notag
& =   \int_{2/3}^1 \left(\frac{3x^4}{1+x^3}x\right)^2 g(x) dx := D_3<\infty.
\end{align}
Notice that $D_3$ is independent of $y$, and therefore
\begin{align}\label{eq:f}
\frac{1}{\sigma_{f[y]}^2} \leq \frac{1}{D_3}.
\end{align}
From Eqs.~(\ref{eq:vy2}) and (\ref{eq:f}) one concludes that $\limy \left(\sigma_{v[y]}-\frac{1}{\sigma_{f[v]}}\right)=\infty$ and the result follows.

\subsection{Proof of Theorem \ref{thm:convergence_Vn}}
Let\footnote{See the paragraph preceding Eq.~(\ref{U_u.o.c.1}) for the definition of
$\Omega_W$.} $\Omega':=\Omega_\theta\times\Omega_W$. This
probability space is the basis for the proof of Theorem
\ref{thm:convergence_Vn}.
From Eqs.~(\ref{eq:prop1}) and (\ref{Unt1}) it follows
that\footnote{The process $\vphihat$ was defined in Section
\ref{sec:optimal_stopping_time} as $\limn\vphitn =\vphiasf$.}
\begin{align}\label{eq:varphi_uoc}
\limn\vphitn = \vphihat,\;\;\Omega'\text{-u.o.c.}
\end{align}



\subsubsection{\textbf{Proof of Eq.~(\ref{eq:convergence_tautnD_to tautD})}}
By using this convergence we show now that on this probability space
$\limn\tau_D^n(\pi)=\tau_D(\pi)$, $\Omega'$-a.s.

\begin{lem}\label{lem:convergence_stopping_times}
Fix $T>0$. Under Assumptions \ref{asu:t_u_t_v},
\ref{asu:expectation1}, \ref{asu:lambda_rates}, \ref{asu:density1},
\ref{asu:density22}, \ref{asu:value_functions}, and \ref{asu:D},
$$\limn (\tautnD(\pi)\wedge T) = (\tautD(\pi)\wedge T),
\;\;\Omega'\text{-{\rm a.s.}}$$
\end{lem}
\begin{proof}
If $\pi\notin D$, then $\tautnD=\tautD=0$. For the case $\pi\in D$
we express the stopping times $\tautnD$ and $\tautD$ in a more
convenient way.
From Eqs.~(\ref{frac_q_l_q_m}) and (\ref{frac_p_l_p_m}) it follows
that
\begin{align}\label{eq:ptA_Xnt}
\frac{\pt^n(t)}{1-\pt^n(t)} =\frac{\pi}{1-\pi}\vphit^n(t)
\end{align}
and
\begin{align}\label{eq:ptt_Xt}
\frac{\phat(t)}{1-\phat(t)} =\frac{\pi}{1-\pi}\vphihat(t).
\end{align}
%
Eqs.~(\ref{eq:ptA_Xnt}) and (\ref{eq:ptt_Xt}) imply that
\begin{align}\label{eq:tautD_with_Xt}
\tautD (\pi)
&=\inf\left\{t \left|  \vphihat(t)\notin
\cup_j\left(c_j,d_j \right)\right.\right\} = \inf\left\{t \left|
\vphihat(t)\notin
\left(c_i,d_i \right)\right.\right\}
\end{align}
and
\begin{align}
\label{eq:tautnD_with_Xtn}
\tautnD (\pi)&=\inf\left\{t \left| \vphit^n(t)\notin
\cup_j\left(c_j,d_j \right)\right.\right\},
\end{align}
where for every index $j$,
$c_j:=\frac{1-\pi}{\pi}\cdot\frac{a_j}{1-a_j}$ and
$d_j:=\frac{1-\pi}{\pi}\cdot\frac{b_j}{1-b_j}$. The second equality
in Eq.~(\ref{eq:tautD_with_Xt}) follows since $(\vphihat(t))$ is a
continuous process w.r.t.~the parameter $t$ (see Eq.~(\ref{vphit})).
In order to prove that $\limn (\tautnD(\pi) \wedge T)=
(\tautD(\pi)\wedge T)$, $\Omega'$-a.s., we distinguish between two
possibilities: $\tautD(\pi) (\omega) > T $ and $\tautD(\pi) (\omega)
\leq T $. Fix\footnote{The following properties that we state hold
for almost every $\omega\in\Omega'$. We chose $\omega\in\Omega'$ for
which these properties hold.} $\omega\in\Omega'$.

If $\tautD(\pi) (\omega) > T $ then, since $\vphihat(t)(\omega)$ is
continuous w.r.t.~$t$, it follows that the supremum
$$M(T)(\omega):=\underset{0\leq t\leq
T}{\sup}\,\vphihat(t)(\omega)$$ and the infimum
$$m(T)(\omega):=\underset{0\leq t\leq
T}{\inf}\vphihat(t)(\omega),$$ are attained and satisfy $c_i <
m(T)(\omega) < M(T)(\omega) < d_i$. Moreover,
Eq.~(\ref{eq:varphi_uoc}) implies that for every $0 < \delta < \min
\{d_i - M(T)(\omega),m(T)(\omega)-c_i\}$ there exists $N_\delta>0$
such that for every $n > N_\delta$ and every $t\in[0,T]$ one has
$$\left|\vphit^n(t)(\omega)-
\vphihat(t)(\omega)\right|<\delta,$$ and therefore,
$$c_i < \vphit^n(t)(\omega) < d_i.$$
Hence, $\tautnD (\omega) > T $, and consequently $(\tautnD
(\omega)\wedge T) = (\tautD (\omega) \wedge T)$.

If $\tautD(\pi) (\omega) \leq T $ we assume without loss of
generality\footnote{The proof for $\vphihat(\tautD(\pi))(\omega) =
c_i$ is similar and is therefore omitted.} that
$\vphihat(\tautD(\pi))(\omega) = d_i$.
Fix $\epsilon>0$. Denote $$\delta_1=\delta_1(\omega):= d_i
-\underset{0\leq t\leq \tautD
(\omega)-\epsilon}{\sup}\,\vphihat(t)(\omega)$$ and
$$\delta_2=\delta_2(\omega):= \underset{0\leq t\leq
\tautD (\omega)-\epsilon}{\inf}\vphihat(t)(\omega) - c_i.$$ By the
continuity of $\vphihat(t)(\omega)$ with respect to $t$, and by the definition of $\tautD$ it
follows that $\delta_1 , \delta_2 > 0$. Denote
$$\delta_3= \delta_3(\omega):=  \underset{\tautD (\omega)\leq t \leq
\tautD (\omega)+\epsilon}
{\sup}\,\vphihat(t)(\omega)-d_i.$$ From Eq.~(\ref{vphit}) and the
fluctuations of the Brownian motion, it follows that $\delta_3 > 0$.
Let
$$\delta_4 :=\frac{1}{2}\left(c_{i+1} - d_i\right) .$$ Assumption \ref{asu:D} implies
that $\delta_4>0$. Denote
$$
\tau_d (\pi)
:=\inf\left\{ t \left| \vphihat(t)=
d + (\delta_4\wedge\delta_3)\right.\right\}.
$$
Then clearly one has $\tautD(\omega)<\tau_d(\omega) <
\tautD(\omega)+\epsilon$. Let $\delta:=\min\{\delta_1, \delta_2,
\delta_3, \delta_4, \epsilon\}$. From Eq.~(\ref{eq:varphi_uoc}) it
follows there exists $N_\delta>0$ such that for every $ n >
N_\delta$ and every $t\in[0,T]$,
$$\left|\vphit^n(t)(\omega) -
\vphihat(t)(\omega)\right|<\delta.$$ Therefore, for every such
$n>N_\delta$ and every $t\in[0,\tautD (\omega)-\epsilon]$ one has
\begin{align}\label{eq:ci}
c_i < \vphit^n(t)(\omega) < d_i,
\end{align}
and at time $\tau_d$ one has
\begin{align}\label{eq:di}
d_i < \vphit^n(\tau_d)(\omega) <d_{i+1}.
\end{align}
Since $\tautD(\omega)<\tau_d(\omega) < \tautD(\omega)+\epsilon$,
Eqs.~(\ref{eq:ci})--(\ref{eq:di}) yield
$$|(\tautnD (\omega)\wedge
T)-(\tautD (\omega)\wedge T)|<\epsilon.$$
\end{proof}


As a corollary, we get that $\limn \tautnD(\pi) \eqd \tautD(\pi)$. This completes the proof of Eq.~(\ref{eq:convergence_tautnD_to tautD}).

\subsubsection{\textbf{Proof of Eq.~(\ref{eq:convergence_Vtautn_to Vtaut_in ecpectation})}}
To avoid cumbersome notation we write $\taut:=\tautD$ and
$\tautn:=\tautnD$. In order to prove that $V^n_{\tautn}(\pi) $
converges to $V_{\taut}(\pi)$ we will bound the expression
$|V^n_{\tautn}(\pi) - V_{\taut}(\pi)|$ by other terms for which the
convergence is easier to prove.
By the triangle inequality, for every index $n\inN$ and every time
$T>0$,
\begin{align}\notag
&\left|V^n_{\tautn}(\pi) - V_{\taut}(\pi)\right|\\\label{eeeqq1}
&\quad\leq  \left|E^\pi\left[
\int_0^{\tautn}re^{-rt}k^n(\pt^n(t))dt
- \int_0^{\tautn\wedge T}re^{-rt}k^n(\pt^n(t))dt
\right]\right|\\\label{eeeqq2}
&\qquad+  \left|E^\pi\left[
\int_0^{\tautn\wedge T}re^{-rt}k^n(\pt^n(t))dt
- \int_0^{\taut\wedge T}re^{-rt}k^n(\pt^n(t))dt
\right]\right|\\\label{eeeqq3}
&\qquad+  \left|E^\pi\left[
\int_0^{\taut\wedge T}re^{-rt}k^n(\pt^n (t))dt
- \int_0^{\taut\wedge T}re^{-rt}k(\phat (t))dt
\right]\right|\\\label{eeeqq4}
&\qquad+ \left|E^\pi\left[
\int_0^{\taut\wedge T}re^{-rt}k(\phat (t))dt
- \int_0^{\taut}re^{-rt}k(\phat (t))dt
\right]\right|\\\label{eeeqq7}
&\qquad+ \left|E^\pi\left[
re^{-r\tautn}\tfrac{1}{n}K^n(\pt^n (\tautn))
- re^{-r\taut}K(\phat (\taut))
\right]\right|
\end{align}
We now show that each of the terms converges to zero. That is, for
every fixed $\epsilon>0$, there exists $N_\epsilon>0$ such that for
every $n>N_\epsilon$, each of the terms is bounded by $\epsilon$. We
divide the proof into four parts.
$\\ $\textbf{Part I: First and fourth terms.}
In this part we show that for every $n\inN$ and for sufficiently
large $T$, if the DM cannot operate the system after time $T$, then
his expected loss by using the stopping time $(\tautn\wedge T)$ is
close up to $\epsilon$ to the expected loss from the integral cost
part without the limitation of the maximal time of operating the
system.
From Remark \ref{rem:k_K_bounded} it follows that the sequence
$\{k^n(\ptn (t))\}_{n\inN}$ is bounded by $C_k$. Therefore, for
every $T>0$,
\begin{align}\label{eq:difference_risk_n_and_risk_Tn}
&\left|E^\pi\left[
\int_0^{\tautn}re^{-rt}k^n(\pt^n (t))dt
- \int_0^{\tautn\wedge T}re^{-rt}k^n(\pt^n (t))dt
\right]\right|\\\notag
&\quad\leq E^\pi\left|\int_{\tautn\wedge T}^ {\tautn}{re^{-rt}k^n(\pt^n
(t))dt } \right|\\\notag
&\quad\leq C_k E^\pi\left|\int_{\tautn\wedge T}^ {\tautn}{re^{-rt} dt }
\right|\\\notag
&\quad\leq C_k E^\pi\left|\int_{T}^ {\infty}{re^{-rt} dt } \right| =
C_ke^{-rT}.
\end{align}
The last term on Eq.~(\ref{eq:difference_risk_n_and_risk_Tn})
converges to zero as T goes to infinity, and so there exists a
constant $T:=T_\epsilon$ such that for every $n\inN$
\begin{align}\notag
\left|E^\pi\left[
\int_0^{\tautn}re^{-rt}k^n(\pt^n (t))dt
- \int_0^{\tautn\wedge T}re^{-rt}k^n(\pt^n (t))dt
\right]\right|<\epsilon.
\end{align}
Similarly, one can choose $T_\epsilon$ to be such that, in addition,
\begin{align}\notag
\left|E^\pi\left[
\int_0^{\taut}re^{-rt}k(\phat (t))dt
- \int_0^{\taut\wedge T}re^{-rt}k(\phat (t))dt
\right]\right|<\epsilon.
\end{align}
$\\ $\textbf{Part II: Second term.}
We now show that for sufficiently large $n\inN$, by changing the
$\mathcal{F}^{L^n}_{nt}$-adapted stopping time $\tautn$ to the
$\mathcal{F}^{\pt}_t$-adapted stopping time $\taut$, the expected
integral cost does not change by much:

\begin{align}\label{eq:risk_VTn_wrt_limTLn}
&\left|E^\pi\left[
\int_0^{\tautn\wedge T}re^{-rt}k^n(\pt^n (t))dt
- \int_0^{\taut\wedge T}re^{-rt}k^n(\pt^n (t))dt
\right]\right|\\\notag
&\quad\leq E^\pi\left|\int_{\taut\wedge T}^{\tautn\wedge
T}re^{-rt}k^n(\pt^n (t))dt \right| \\\notag
&\quad\leq rC_k  E^\pi\left|({\tautn}\wedge T) - (\taut\wedge T )
\right|.\\\notag
\end{align}
From Lemma \ref{lem:convergence_stopping_times} one has
$\limn(\tautn\wedge T) = (\taut\wedge T)$, $\Omega'$-a.s. Therefore,
by the bounded convergence theorem, there exists $N_\epsilon>0$ such
that for every $n>N_\epsilon$ the last term in
Eq.~(\ref{eq:risk_VTn_wrt_limTLn}) is smaller than $\epsilon$.
$\\ $\textbf{Part III: Third term.}
In this part we show that if the DM cannot operate the system after
time $T$, then his expected integral cost from the $n$-th system
and by using the $\mathcal{F}^{\pt}_t$-adapted stopping time $\taut$ is
close to the expected integral cost of the limit problem using the
same stopping time $\taut$:
\begin{align}\label{eq:third_term}
&\left|E^\pi\left[
\int_0^{\taut\wedge T}re^{-rt}k^n(\pt^n (t))dt
- \int_0^{\taut\wedge T}re^{-rt}k(\pt (t))dt
\right]\right|\\\notag
&\quad\leq E^\pi\left[
\int_0^{\taut\wedge T}re^{-rt}\left|k^n(\pt^n (t))-k(\phat
(t))\right|dt
\right]\\\notag
&\quad\leq E^\pi\left[
T \underset{0\leq t\leq T}{\sup}\,\left|k^n(\pt^n (t))-k(\phat
(t))\right|
\right].
\end{align}
From Eqs.~(\ref{eq:varphi_uoc}), (\ref{eq:ptA_Xnt}), and
(\ref{eq:ptt_Xt})
 it follows that $\limn\pt^n (t)=\phat (t)$,
$\Omega'$-u.o.c. Moreover, by Assumption
\ref{asu:value_functions}.1, the functions $k^n$ converge uniformly
on $[0,1]$ to $k$, and therefore, $\limn \underset{0\leq t\leq
T}{\sup}\,\left|k^n(\pt^n (t))-k(\phat (t))\right|=0$, $\Omega'$-a.s.
The bounded convergence theorem implies that for sufficiently large
$n\inN$, the last term in Eq.~(\ref{eq:third_term}) is smaller than
$\epsilon$.

$\\ $\textbf{Part IV: Fifth term.} In this part we show that for
sufficiently large $n\inN$, the expected terminal cost from the
$n$-th system using the stopping time $\tautn$ is relatively close
to the expected terminal cost from the limit system using the
stopping time $\taut$. To this end, we show that
$\limn re^{-r\tautn} \tfrac{1}{n}K^n(\pt^n (\tautn)) = re^{-r\taut }K(\phat
(\taut))$, $\Omega'$-a.s. From Remark \ref{rem:k_K_bounded} and the
bounded convergence theorem it will follow that there exists
$N_\epsilon >0 $ such that, for every $n>N_\epsilon$,
$$\left|E^\pi\left[
re^{-r\tautn}\tfrac{1}{n}K^n(\pt^n (\tautn))
- re^{-r\taut}K(\phat (\taut))
\right]\right|<\epsilon.$$
From Lemma \ref{lem:convergence_stopping_times} it follows that
\begin{align}
P(\omega\in\Omega' \mid \forall T\in\mathbb{N}\;\;\limn
(\tautn(\pi)(\omega)\wedge T) = (\taut(\pi)(\omega)\wedge T))=1.
\end{align}
Fix $\omega\in\Omega'$ such that for every $T\in\mathbb{N}$ one has
$\limn (\tautn (\omega)\wedge T) = (\taut(\omega)\wedge T)$ and
$\limn \pt^n (\omega) = \phat  (\omega)$ u.o.c. We divide the proof
into two cases: $\taut(\omega)=\infty$ and $\taut(\omega)<\infty$.
If $\taut(\omega)=\infty$ then, since by Lemma
\ref{lem:convergence_stopping_times} one has $\limn (\tautn
(\omega)\wedge T) = (\taut(\omega)\wedge T) = T$, it follows that
there exists $N_\epsilon>0$ such that, for every $n>N_\epsilon$,
$|(\tautn (\omega)\wedge T) - T|<1$. Let $T$ be such
that $re^{-r(T-1)}<\frac{\epsilon}{C_K}$. Then for every
$n>N_\epsilon$,
\begin{align}\notag
&|re^{-r\tautn} \tfrac{1}{n}K^n(\pt^n (\tautn))(\omega) - re^{-r\taut} K(\phat
(\taut))(\omega) |\\\notag
&\quad= |re^{-r\tautn}  \tfrac{1}{n}K^n(\pt^n (\tautn))(\omega)  |\\\notag
&\quad\leq C_K
re^{-r(T-1)} \leq\epsilon.
\end{align}
If $\taut(\omega)<\infty$ then, since by Lemma
\ref{lem:convergence_stopping_times}  $\limn\tautn (\omega) =
\taut(\omega)$,  $\Omega'$-a.s., it follows that for sufficiently
large $n\inN$ the following two conditions hold:
\begin{align}\label{eq:conditions_tautn1}
\tautn (\omega)&<\taut(\omega)+1, \\\label{eq:conditions_tautn2}
|e^{-r\tautn(\omega)} -
e^{-r\taut(\omega)}|&<\frac{\epsilon}{2rC_K}.
\end{align}
By Assumptions \ref{asu:value_functions}.3 and
\ref{asu:value_functions}.4, the functions $K^n/n$ converge uniformly
on $[0,1]$ to the continuous function $K$. Since $\limn\pt^n (t)
(\omega) = \phat (t)(\omega)$ uniformly on $[0, \taut(\omega) +1]$, it
follows from Eq.~(\ref{eq:conditions_tautn1}) that for sufficiently
large $n\inN$
\begin{align}\label{eq:conditions_tautn3}
| \tfrac{1}{n}K^n(\pt^n (\tautn))(\omega) - K(\phat
(\taut))(\omega)|<\frac{\epsilon}{2rC_K}.
\end{align}
By combining
Eqs.~(\ref{eq:conditions_tautn2})--(\ref{eq:conditions_tautn3}) one
concludes that there exists $N_\epsilon>0$ such that for every
$n>N_\epsilon$,
\begin{align}\notag
&|re^{-r\tautn}  \tfrac{1}{n}K^n(\pt^n (\tautn)) - re^{-r\taut }K(\phat
(\taut))|(\omega) \\\notag
&\quad\leq re^{-r\tautn(\omega)}| \tfrac{1}{n}K^n(\pt^n (\tautn))(\omega) - K(\phat
(\taut))(\omega)|
+
|K(\phat (\taut))(\omega)|
|re^{-r\tautn(\omega)}-re^{-r\taut(\omega)} | \leq \epsilon.
\end{align}

This completes the proof of  Eq.~(\ref{eq:convergence_Vtautn_to Vtaut_in ecpectation}).




\bibliographystyle{plain} 
\bibliography{bib_Asaf} 

\end{document}